\documentclass[12pt]{amsart} 
\usepackage{amscd} 
\usepackage{amsmath}
\usepackage{amssymb}
\usepackage{xypic}

\DeclareMathOperator{\ad}{ad}
\DeclareMathOperator{\cof}{cofiber}

\DeclareMathOperator{\Free}{Free}
\DeclareMathOperator{\GL}{GL}

\DeclareMathOperator{\map}{map}

\DeclareMathOperator{\Sp}{Sp}

\def\Under#1#2{%
   \setbox0=\hbox{$#1$}%
   \setbox1=\hbox to \wd0{\hfil}%
   \dp1=\dp0%
   \hbox to 0pt{$#1$\hss}%
   \raise #2\dp0\hbox{\underline{\box1}}}

\newcommand{\epi}{\twoheadrightarrow}
\newcommand{\Acal}{{\mathcal{A}}}
\newcommand{\Ccal}{{\mathcal{C}}}
\newcommand{\Dcal}{{\mathcal D}}
\newcommand{\Ecal}{{\mathcal E}}

\newcommand{\Fcal}{{\mathcal{F}}}
\newcommand{\field}{{\mathbb{F}}}
\newcommand{\hobased}{{\tilde{h}}}
\newcommand{\hocolim}{\operatorname{hocolim}\,}
\newcommand{\colim}{\operatorname{colim}\,}
\DeclareMathAlphabet{\mathitbf}{OML}{cmm}{b}{it}
\newcommand{\kay}{{\mathitbf{k}}}

\newcommand{\kund}{ {\underline{k}} }
\newcommand{\Kcal}{{\mathcal{K}}}
\newcommand{\Lcal}{{\mathcal{L}}}

\newcommand{\Mbar}{{\overline{M}}}
\newcommand{\n}{ {\underline{n}} }

\newcommand{\nontrivial}{\rm{ntrv}}

\newcommand{\one}{ {\underline{1}} }
\newcommand{\Pcal}{{\mathcal{P}}}
\newcommand{\q}{ \Under{q}{.5} }
\newcommand{\Rcal}{{\mathcal{R}}}
\newcommand{\Sets}{{\mathcal{S}}}
\newcommand{\Sphere}{{\bf{S}}}

\newcommand{\BAR}{\operatorname{Bar}}
\newcommand{\EM}{\operatorname{H}}
\newcommand{\HZ}{\EM\integers}

\newcommand{\MacLane}{Mac\,Lane }

\newcommand{\complexes}{{\mathbb{C}}}
\newcommand{\integers}{{\mathbb{Z}}}
\newcommand{\naturals}{{\mathbb{N}}}
\newcommand{\reals}{{\mathbb{R}}}

\numberwithin{equation}{section}

\newtheorem{theorem}[equation]{Theorem}
\newtheorem{lemma}[equation]{Lemma}
\newtheorem{proposition}[equation]{Proposition}
\newtheorem{corollary}[equation]{Corollary}
\newtheorem{conjecture}[equation]{Conjecture}

\theoremstyle{definition}
\newtheorem{definition}[equation]{Definition}
\newtheorem{remark}[equation]{Remark}

\newtheorem{example}[equation]{Example}

\newtheorem*{NullMapProposition}
       {Proposition~\ref{prop: null map}}
\newcommand{\NullMapPropositionText}[1]{
  In #1, the map 
  $\Sigma^{-1}L(k)\rightarrow T(k)$ is null-homotopic. Thus
  $\Mbar(k)$ splits as a wedge sum $\Mbar(k)\simeq T(k)\vee L(k)$.
}

\newtheorem*{MbarProjectiveProposition}
       {Proposition~\ref{prop: Mbar projective}}
\newcommand{\MbarProjectivePropositionText}{
$\Mbar(k)$ is the stable wedge summand of
$\Sigma^{\infty}(B\Gamma_{k})_{+}$ at the prime $p$
corresponding to the symplectic Steinberg idempotent
in the group ring 
$\field_{p}\left[\Sp_{2k}\left(\field_{p}\right)\right]$. 
Thus $\Mbar(k)$ is a projective spectrum, in the sense of Kuhn. 
}

\newtheorem*{ConnectivityProposition}
       {Proposition~\ref{prop: Rognes connectivity conjecture}}
\newcommand{\ConnectivityPropositionText}{
  The subquotient spectrum $F_{m}bu/F_{m-1}bu$ of the stable rank filtration
  of Rognes is contractible unless
  $m=p^k$ for some prime $p$. If $m=p^k$, then the bottom nontrivial
  homotopy group of $F_{m}bu/F_{m-1}bu$ occurs in dimension $2m-2$.
}

\newtheorem*{ConnectionToMbar}
       {Corollary~\ref{cor: A-L vis Rognes, revisited}}
\newcommand{\ConnectionToMbarText}{
There is an equivalence
$\Mbar(k)\simeq \Rcal_{p^{k}}\kay\Ccal/\Rcal_{p^{k-1}}\kay\Ccal$. 
}

\newtheorem*{ConnectionTheorem}
       {Theorem~\ref{thm: pushout theorem}}
\newcommand{\ConnectionTheoremText}{
For every $m$, there is a stable homotopy pushout square
of augmented $\Gamma$-spaces
\[
\begin{CD}
\Rcal_{m}\Ccal @>>> \Ccal\\
@VVV @VVV\\
\Sp^{m}   @>>>   \Kcal_{m}\Ccal
\end{CD}
\]\\
and hence a homotopy pushout diagram of spectra
\[
\begin{CD}
\Rcal_{m}\kay\Ccal @>>> \kay\Ccal\\
@VVV @VVV\\
\Sp^{m}(\Sphere) @>>> A_{m}.
\end{CD}
\]
}

\newtheorem*{buWhiteheadConjecture}
        {Conjecture~\ref{conjecture: bu Whitehead reduced}}
\newcommand{\MbarComplex}{
\dots\longrightarrow \Mbar(2)\longrightarrow\Mbar(1)\longrightarrow\Mbar(0)
       \longrightarrow bu
}
\newcommand{\buWhiteheadConjectureTextIntro}{
The complex~\eqref{eq: Mbar complex exact}
is exact at the prime $p$. 
}
\newcommand{\buWhiteheadConjectureText}{
The complex
\[
\MbarComplex
\]
is exact at the prime $p$. 
}

\newtheorem*{propositionpushout}
    {Proposition~\ref{proposition: two pushouts}}

\newcommand{\propositionpushouttext}{
Suppose given a diagram of $\Gamma$-spaces
\begin{equation*}    
F_{1}\longleftarrow F_{0} \longrightarrow F_{2}
\end{equation*}
that is a diagram of augmented monoids, and suppose that
$F_{12}$ is the objectwise homotopy pushout. Then the natural map 
\[
F_{12}\longrightarrow \BAR(F_{1}, F_{0}, F_{2})
\]
is a strong augmented stable equivalence.}

\newtheorem*{propositionequivalentcats}
     {Proposition~\ref{proposition: equivalent cats}}
\newcommand{\propositionequivalentcatstext}
{
  Let $F\rightarrow G$ be an augmentation-preserving
  map of augmented very special $\Gamma$-spaces, and suppose that 
  $\Rcal_{i}[F(\one)]\rightarrow\Rcal_{i}[G(\one)]$
  is a homotopy equivalence for all $i<m$. Then 
$\Rcal_{i}[F(X)]\rightarrow\Rcal_{i}[G(X)]$ is a homotopy equivalence
for all $i<m$ and all $X$, and the commuting diagram
\begin{equation*}
\begin{CD}     
\Rcal_{m}[F(\one)]\wedge X @>>> \Rcal_{m} [F(X)]\\
@VVV @VVV\\
\Rcal_{m}[G(\one)]\wedge X @>>> \Rcal_{m} [G(X)]
\end{CD}
\end{equation*}
is a strong homotopy pushout diagram of augmented $\Gamma$-spaces,
that is, it remains a homotopy pushout square after the application of
$\Rcal_{j}$ for all $j$.
}

\newtheorem*{PropositionEquivalence}
     {Proposition~\ref{prop: equivalence of Rognes filtrations}}
\newcommand{\PropositionEquivalenceText}
{Let $\Ccal$ be the topological category of finite-dimensional 
complex vector spaces, let $\Rcal_{m}\kay\Ccal$ be the modified stable
rank filtration of $\kay\Ccal=bu$, and let
$F_{m}bu$ be the $m$-th stable rank filtration of Rognes as applied to 
complex topological $K$-theory. Then the canonical map of filtrations
$\Rcal_{m}\kay\Ccal\rightarrow F_{m}bu$ is a homotopy
equivalence of spectra for each $m$.  
}

\begin{document}
\baselineskip=17pt

\title[Augmented $\Gamma$-spaces] {Augmented $\Gamma$-spaces, the
  stable rank filtration, and a $bu$-analogue of the Whitehead
  conjecture}

\author[G. Z. Arone]{Gregory Z. Arone}
\address{Kerchof Hall, University of Virginia, P.O. Box 400137,
         Charlottesville VA 22904 USA}
\email{zga2m@virginia.edu}
\urladdr{http://www.math.virginia.edu/~zga2m}
\author[K. Lesh]{Kathryn Lesh}
\address{Department of Mathematics, Union College, Schenectady, NY 12309}
\email{leshk@union.edu}
\urladdr{http://www.math.union.edu/~leshk}

\date{}

\begin{abstract} 
  We explore connections between \cite{A-L}, where we constructed
  spectra that interpolate between $bu$ and $\HZ$, and earlier
  work of Kuhn and Priddy on the Whitehead conjecture and of Rognes on
  the stable rank filtration in algebraic $K$-theory. We construct a
  ``chain complex of spectra'' that is a $bu$-analogue of an auxiliary
  complex used by Kuhn-Priddy; we conjecture that this chain complex
  is ``exact''; and we give some supporting evidence. We tie this to
  work of Rognes by showing that our auxiliary complex can be
  constructed in terms of the stable rank filtration.  As a
  by-product, we verify for the case of topological complex $K$-theory a
  conjecture made by Rognes about the connectivity (for certain rings)
  of the filtration subquotients of the stable rank filtration of
  algebraic $K$-theory.
\end{abstract}



\maketitle

\section{Introduction}
\label{section: introduction}

In \cite{A-L}, we introduced a sequence of spectra $\{A_{m}\}$
interpolating between the connective complex $K$-theory spectrum $bu$
and the integral Eilenberg-\MacLane spectrum $\HZ$.  These new
spectra resulted from a general construction on permutative categories
endowed with an ``augmentation.''  In the current work, we explore
connections of that construction to other settings, in particular to
work of Kuhn and Priddy \cite{Kuhn-Priddy} on the Whitehead
Conjecture, and to work of Rognes \cite{Rognes} on the stable rank
filtration of algebraic $K$-theory.  Connections to Kuhn and Priddy's
work were suggested by the many properties that the spectra $A_{m}$
share with the symmetric powers of the sphere spectrum,
$\Sp^{m}(\Sphere)$, which can also be given as an example of the
categorical construction of \cite{A-L}.  This led us to call $A_{m}$
the ``$bu$-analogue'' of $\Sp^{m}(\Sphere)$ and to propose a
$bu$-analogue of the Whitehead Conjecture. On the other hand, our
construction was also reminiscent of the stable rank filtration in
algebraic $K$-theory, and this made it natural to ask the exact
relationship between the two filtrations for $bu$.  Curiously, the
two threads converged: in this paper we construct a $bu$-analogue of
an auxiliary complex introduced by Kuhn and Priddy in the course of
their proof of the Whitehead conjecture, and it turns out to be
closely related to the stable rank filtration for topological
complex $K$-theory.  As a by-product, we obtain a good understanding of
the stable rank filtration in the case of $bu$; in particular, we are
able to verify in this case the connectivity conjecture made by Rognes in 
the context of Euclidean domains or local rings
(\cite{Rognes}, Conjecture~12.3, though Rognes does not explicitly
mention topological rings as part of his context). 

To describe our results in detail, we need to recall the overall setup
of~\cite{A-L}. Let $\Ccal$ be a permutative category.  To such a
category, Segal's machine~\cite{Segal} associates a spectrum that we
will denote $\kay\Ccal$ in the general case. The simplest example is
$\naturals$, the permutative ``category'' with nonnegative integers as
objects, no nonidentity morphisms, and permutative structure given by
addition. In this case the associated spectrum is the integral
Eilenberg-\MacLane spectrum, which we denote as usual by $\HZ$,
rather than $\kay\naturals$.   A permutative category $\Ccal$ is called
augmented if there is a symmetric monoidal functor 
$\epsilon: \Ccal\rightarrow\naturals$ such that $\epsilon^{-1}(0)$ is
the trivial one-object category. An augmentation induces a map of
spectra $\kay\Ccal\rightarrow \HZ$. For example, if $\Sets$ is
the category of finite pointed sets, with augmentation given by
non-basepoint cardinality, then $\kay\Sets$ is the sphere spectrum
$\Sphere$, and the augmentation induces the Hurewicz map
$\Sphere\rightarrow \HZ$.  Similarly, the category of complex
vector spaces and unitary isomorphisms can be augmented by dimension,
and this augmentation induces a map $bu\rightarrow \HZ$, which
is the first structure map in the tower of connective covers of complex
$K$-theory.

Let $\Ccal$ be an augmented permutative category. The main
construction of~\cite{A-L} associates to $\Ccal$ a sequence of
permutative categories
\begin{equation}         \label{eq: sequence of categories}
\Ccal=\Kcal_0\Ccal\longrightarrow \Kcal_1\Ccal 
                  \longrightarrow \cdots 
                  \longrightarrow \Kcal_{m}\Ccal 
                  \longrightarrow \cdots 
                  \longrightarrow \Kcal_\infty\Ccal\simeq \naturals. 
\end{equation}
The associated spectra and maps refine the map 
$\kay\Ccal\longrightarrow \HZ$,
\begin{equation}   \label{eq: sequence of spectra}
\kay\Ccal \longrightarrow \kay\Kcal_{1}\Ccal 
          \longrightarrow \cdots 
          \longrightarrow \kay\Kcal_{m}\Ccal 
          \longrightarrow \cdots \longrightarrow \HZ.
\end{equation}
A key example is $\Ccal=\Sets$, which gives
$\kay\Kcal_{m}\Ccal\simeq\Sp^{m}(\Sphere)$ and recovers the classical
filtration of $\HZ\simeq\Sp^\infty(\Sphere)$ by the finite
symmetric powers of the sphere spectrum.
The other primary example in \cite{A-L} is the category of
finite-dimensional complex vector spaces and unitary isomorphisms,
augmented by dimension.  In keeping with the notation in \cite{A-L},
we use $A_{m}$ to denote the spectrum $\kay\Kcal_{m}\Ccal$ in the 
complex vector space case. 

Standard manipulations of cofiber sequences allow us to recast a
sequence such as \eqref{eq: sequence of spectra} as a ``chain
complex'' of spectra involving the successive cofibers, and one can
then ask whether this complex is ``exact.''  
(See Section~\ref{sec: construction of Mbar}.) This question at a
prime $p$ for the category of finite pointed sets gives the classical Whitehead
Conjecture, stated below, which was proved by Kuhn for $p=2$ and by
Kuhn and Priddy for odd primes.  When localized at the prime $p$, the
sequence $\Sp^{m}(\Sphere)$ only changes at powers of $p$, and we adopt
the usual notation
\[
L(k) =\Sigma^{-k}\Sp^{p^k}(\Sphere)/\Sp^{p^k-1}(\Sphere). 
\]
Note that $L(0)=S^{0}$, and that in the following theorem, the map
$L(0)\rightarrow \HZ$ is the Hurewicz map. 

\begin{theorem}[Whitehead Conjecture, \cite{Kuhn, Kuhn-Priddy}]
\label{theorem: classical Whitehead}
The complex
\begin{equation}  \label{eq: unreduced Whitehead}
\dots\longrightarrow L(2)\longrightarrow L(1)\longrightarrow L(0)
       \longrightarrow \HZ
\end{equation}
is exact at the prime $p$. 
\end{theorem}

Because \cite{A-L} established striking similarities between the
subquotients $A_{m}/A_{m-1}$ and $\Sp^{m}(\Sphere)/\Sp^{m-1}(\Sphere)$
(that is, between the subquotients of \eqref{eq: sequence of spectra}
for finite-dimensional complex vector spaces and
finite pointed sets, respectively), we were led to conjecture a
$bu$-analogue of the Whitehead Conjecture.  Again working at a prime $p$,
we let
\[
T(k)=\Sigma^{-(k+1)}A_{p^k}/A_{p^k-1}.
\]

\begin{conjecture}[$bu$ Whitehead Conjecture]
\label{conjecture: bu Whitehead}
The complex
\begin{equation}   \label{eq: bu whitehead complex}
\dots\longrightarrow T(2)\longrightarrow T(1)\longrightarrow T(0)
       \longrightarrow bu \longrightarrow \HZ
\end{equation}
is exact at the prime $p$. 
\end{conjecture}
In considering the potential for adapting Kuhn and Priddy's methods to
prove Conjecture~\ref{conjecture: bu Whitehead}, we
see that Kuhn and Priddy did not actually work directly
with the complex $\{L(k)\}$ in the proof of 
Theorem~\ref{theorem: classical Whitehead}. Instead, they constructed an
auxiliary complex $\{M(k)\}$ defined by
$M(k)=\Sigma^{-k}D(k)/D(k-1)$, where $D(k)$ is the cofiber of the
$p$-fold diagonal map
$\Sp^{p^{k-1}}(\Sphere)\rightarrow\Sp^{p^{k}}(\Sphere)$. They
proved that 
\begin{equation} \label{eq: mod p Whitehead}
\dots\longrightarrow M(2)\longrightarrow M(1)\longrightarrow M (0)
       \longrightarrow \HZ/p
\end{equation}
is exact at the prime $p$
(the ``mod $p$ Whitehead Conjecture'') and then used this
to obtain Theorem~\ref{theorem: classical Whitehead}.

In this paper, we define spectra $\Mbar(k)$ that may play a similar
role in a $bu$ version of the Whitehead Conjecture to the role played
by $M(k)$ in the classical Whitehead Conjecture. The first part of the
paper constructs the spectra $\Mbar(k)$, studies their properties, and
gives evidence for the exactness of the complex $\{\Mbar(k)\}$
based on the calculus of functors. The second part of the paper sets
out a fairly general categorical construction related to the rank
filtration and explores some of its technical properties. The last
part of the paper shows that $\Mbar(k)$ bears a close relationship to
the categorical construction introduced in the second part,
as well as to the $k$-th filtration quotient of the stable rank filtration of
Rognes (if one were to apply the stable rank filtration to topological
complex $K$-theory). In the remainder of this introduction, we
summarize each of these segments and highlight the new results.

The spectra $\Mbar(k)$ are constructed in 
Section~\ref{sec: construction of Mbar}. We use the fact that 
the sequence~\eqref{eq: sequence of categories} is natural in augmented
permutative categories, and we consider the functor from finite pointed sets
to finite-dimensional complex vector spaces that takes a set $S$ with
basepoint $*$ to $\complexes(S)/\complexes(*)$. This functor respects
the augmentation and induces maps of spectra
$\Sp^{m}(\Sphere)\rightarrow A_{m}$ compatible with the inclusions
$\Sp^{m-1}(\Sphere)\rightarrow \Sp^{m}(\Sphere)$ and
$A_{m-1}\rightarrow A_{m}$. Further, the two sequences of spectra
share the property that $\Sp^{m-1}(\Sphere)\rightarrow
\Sp^{m}(\Sphere)$ and $A_{m-1}\rightarrow A_{m}$ are equivalences at a
prime $p$ unless $m$ is a power of $p$. We index logarithmically and
denote the cofiber of $\Sp^{p^{k}}(\Sphere)\rightarrow A_{p^{k}}$ by
$C(k)$, by analogy to Kuhn and Priddy's $D(k)$, and we define
\begin{equation}   \label{eq: defn Mbar(k)}
\Mbar(k) =\Sigma^{-(k+1)}C(k)/C(k-1),
\end{equation}
by analogy to Kuhn and Priddy's $M(k)$.  The standard manipulation
fits these spectra together into a chain complex,
\begin{equation}    \label{eq: Mbar complex exact}
\MbarComplex.
\end{equation}

The following conjecture is analogous to the exactness 
of~\eqref{eq: mod p Whitehead}.
\begin{buWhiteheadConjecture}[Reduced $bu$ Whitehead Conjecture]
\buWhiteheadConjectureTextIntro
\end{buWhiteheadConjecture}

There are compelling similarities between the spectra $M(k)$ 
of~\eqref{eq: mod p Whitehead} and the spectra $\Mbar(k)$ 
of~\eqref{eq: Mbar complex exact}.  As a first example,
we compare the relationship of~\eqref{eq: unreduced Whitehead} 
and~\eqref{eq: mod p Whitehead} in the classical case with
the relationship of \eqref{eq: bu whitehead complex}
and \eqref{eq: Mbar complex exact} in the
$bu$ case. In the classical case, $M(k)$ is a cofiber of subquotients of
symmetric powers of the sphere spectrum, and it sits in a cofiber
sequence
\[
\Sigma^{-1}L(k-1)\rightarrow L(k)\rightarrow M(k).
\]
(See \eqref{eq: symm power chain complex}.)
It turns out that the first map is null-homotopic, and so $M(k)$
splits as a wedge sum, $M(k)\simeq L(k)\vee L(k-1)$
\cite{Mitchell-Priddy}. Similarly in the $bu$ situation, $\Mbar(k)$
is a cofiber of subquotients of symmetric powers of the sphere spectrum
and subquotients of the spectra $A_{p^{k}}$ filtering $bu$,
and there is a cofiber sequence
\begin{equation}    \label{eq: cofiber in bu}
\Sigma^{-1}L(k)\rightarrow T(k)\rightarrow \Mbar(k).
\end{equation}
(See \eqref{eq: bu chain complexes}.)
We prove the following proposition.

\begin{NullMapProposition} 
  \NullMapPropositionText{\eqref{eq: cofiber in bu}}
\end{NullMapProposition}

A second characteristic of $M(k)$ that is important in Kuhn and
Priddy's work is that the Steinberg idempotent in 
$\field_{p}\left[\GL_{k}\left(\field_{p}\right)\right]$
splits $M(k)$ off from
$\Sigma^{\infty}\left(B\Delta_{k}\right)_{+}$, where
$\Delta_k\cong(\integers/p)^k$ is a transitive elementary abelian
$p$-subgroup of $\Sigma_{p^{k}}$. Being a wedge summand of a
suspension spectrum makes $M(k)$ a projective spectrum
in the sense of Kuhn's homological algebra of spectra. Also,
because $M(k)$ naturally splits from the suspension spectrum of a
classifying space (rather than a Thom space, as for $L(k)$), it is
easier to apply techniques involving the transfer. We establish a
similar result for our spectra $\Mbar(k)$.  The group that corresponds
to $\Delta_{k}$ in the unitary situation is the irreducible projective
elementary abelian $p$-subgroup $\Gamma_{k}$ of $U(p^{k})$, described
in \cite{A-L} Section 10. 

\begin{MbarProjectiveProposition}
\MbarProjectivePropositionText
\end{MbarProjectiveProposition}



While the spectra $\Mbar(k)$ may seem at first to be an ad hoc
construction, it turns out that they are an example of a general
categorical construction closely related to Rognes's stable rank
filtration of algebraic $K$-theory. We set up this construction in
detail in Section~\ref{section: filtered Gamma spaces}, but we summarize 
it here. Recall that a
$\Gamma$-space is a pointed functor from the category of finite
pointed sets to the category of pointed simplicial sets.  An important
example of a $\Gamma$-space is given by the infinite symmetric product
functor $\Sp^\infty$, whose stabilization is the integral
Eilenberg-\MacLane spectrum $\HZ$. We say that a $\Gamma$-space $F$ is
``augmented'' if it comes equipped with a natural transformation
$\epsilon: F\rightarrow \Sp^\infty$ such that $\epsilon^{-1}(\Sp^0)=*$
(Definition~\ref{defn: augmentation}). Then we give $F$ a natural
filtration $\Rcal_{m}F$ by pulling back the filtration of $\Sp^{\infty}$ 
by finite symmetric powers. If $F$ is obtained by applying Segal's 
construction to a permutative augmented category $\Ccal$ 
(Definition~\ref{defn: augmented category}), then we denote this
filtration $\Rcal_{m}\Ccal$ and call it the ``modified rank
filtration'' of the $\Gamma$-space associated to $\Ccal$. 
There is a corresponding filtration of the
$K$-theory spectrum of $\Ccal$, and we call this the ``modified stable
rank filtration'' of the $K$-theory spectrum $\kay\Ccal$ 
(Definition~\ref{defn: modified stable rank filtration}).  It differs
from the original stable rank filtration of algebraic $K$-theory given
by Rognes in taking place within the framework of Segal's
$\Gamma$-spaces rather than Waldhausen's $S_\bullet$ construction, but
it is otherwise similar in spirit.




In Sections~\ref{sec: sums and products} 
and~\ref{section: very special Gamma spaces}, we do some technical
work to prepare for the main goal of 
Section~\ref{section: comparison of constructions}, namely to exhibit
$\Mbar(k)$ as a particular example of the modified stable rank
filtration. We say that
an augmentation-preserving map $F\rightarrow G$ of augmented
$\Gamma$-spaces is a ``strong augmented (stable) equivalence''
if it induces a (stable) equivalence $\Rcal_{m}F\rightarrow\Rcal_{m}G$
at each level $m$ of the filtration associated to the augmentation.
The goal of Section~\ref{sec: sums and products} 
is to prove the proposition below, establishing that a bar
construction  of appropriate $\Gamma$-spaces is strongly equivalent to
an objectwise homotopy pushout;
this gives a tool to understand the augmentation-induced filtration
of $A_{m}$, which is defined in \cite{A-L} by a bar construction. 

\begin{propositionpushout}
\propositionpushouttext
\end{propositionpushout}


In Section~\ref{section: very special Gamma spaces}, we study filtered
$\Gamma$-spaces that are ``very special,'' in the sense of Segal. The
primary example we have in mind is, of course, the $\Gamma$-space
associated to a permutative augmented category, because we want to
relate the modified stable rank filtration in this case to the spectra
$\Mbar(k)$ in the reduced $bu$ Whitehead Conjecture. Given a map between
augmented, very special $\Gamma$-spaces, the main result of
Section~\ref{section: very special Gamma spaces} allows us to identify
the first place where the filtrations differ, and to describe the 
difference in terms of the value of the $\Gamma$-spaces on $S^{0}$.

\begin{propositionequivalentcats}  
\propositionequivalentcatstext
\end{propositionequivalentcats}

In the final part of the paper, we reach our goal of tying the spectra
$\Mbar(k)$, defined in Section~\ref{section: Mbar} in terms of
\cite{A-L}, to the stable rank filtration of Rognes. First, in
Section~\ref{section: comparison of constructions}, 
we study the relationship
between the modified stable rank filtration and the construction
of \cite{A-L}. 
Given an augmented permutative category $\Ccal$, 
the modified stable rank filtration is a filtration
from $*$ to $\kay\Ccal$, while the spectra $A_{m}$ arising from the 
construction of \cite{A-L} applied to $\Ccal$ 
give a filtration going from $\kay\Ccal$ to $H\integers$.  
The following theorem gives the relationship. 

\begin{ConnectionTheorem}
\ConnectionTheoremText
\end{ConnectionTheorem}

As an immediate consequence, we find that when $\Ccal$ is the category
of finite-dimensional complex vector spaces, the subquotients of the
stable rank filtration are actually the same as the spectra
$\Mbar(k)$. 

\begin{ConnectionToMbar}
\ConnectionToMbarText
\end{ConnectionToMbar}

Lastly, in Section~\ref{section: comparison modified to Rognes} we
address the relationship of the modified stable rank filtration and
the original stable rank filtration defined by Rognes. Consider the
category of finitely generated free modules over a nice commutative
ring $R$, which is the context of~\cite{Rognes} (where, however, the
topological case is not explicitly considered). The associated
spectrum is the free $K$-theory spectrum of $R$.  We include here the
case where $R$ is $\reals$ or $\complexes$ by considering the
topologically enriched category of real or complex vector spaces, with
associated spectrum $bo$ or $bu$.  In Section~\ref{section: comparison
  modified to Rognes} we show that the natural map from Segal's
construction of $K$-theory to Waldhausen's, described in Section~1.8
of~\cite{Waldhausen}, induces a map from the modified stable rank
filtration to the original stable rank filtration of Rognes.  The
following proposition establishes that the map is in fact an
equivalence of filtrations in the special case of topological
$K$-theory. We note, however, that the equivalence comes about because
the map on the associated graded is the inclusion of block diagonal
matrices into block upper triangular matrices, which is a homotopy
equivalence over the complex numbers, but not generally for discrete rings.
Hence the stable rank filtration and the modified stable rank filtration 
will typically be different for discrete rings.

\begin{PropositionEquivalence}
\PropositionEquivalenceText
\end{PropositionEquivalence}

The proposition, together with
Theorem~\ref{thm: pushout theorem} and the analysis
in~\cite{A-L}, yields a good understanding of the stable rank
filtration of $bu$.  Let $F_{m}bu$ be the $m$-th stage in the original
stable rank filtration of $bu$. 
It follows from Proposition~\ref{prop: equivalence of Rognes filtrations},
Theorem~\ref{thm: pushout theorem}, and Proposition~\ref{prop: null map}
that there is an equivalence
\[
F_{m}bu/F_{m-1}bu \simeq 
     \Sigma^{-1} A_m/A_{m-1}\vee \Sp^{m}(\Sphere)/\Sp^{m-1}(\Sphere).
\]
We use this to prove the following
result about the connectivity of the subquotients of this filtration,
which confirms for the special case of topological complex $K$-theory
the general conjecture made by Rognes in~\cite{Rognes} (Conjecture~12.3)
for certain discrete rings.

\begin{ConnectivityProposition}
\ConnectivityPropositionText
\end{ConnectivityProposition}

We conclude this introduction with an open-ended remark about the
possible minimality of the various filtrations mentioned here.
In~\cite{A-L} we constructed the complex of spectra $T(k)$ over the
fiber of $bu\rightarrow H\integers$, while in this paper we
construct the complex of spectra $\Mbar(k)$ over the spectrum $bu$ itself. 
It is a consequence of Proposition~\ref{prop: null map} 
that the complex $\{T(k)\}$ is strictly smaller than the complex 
$\{\Mbar(k)\}$, which arises from the stable rank filtration
(by Corollary~\ref{cor: A-L vis Rognes, revisited}).
In~\cite{A-L} we conjectured that
the complex $\{T(k)\}$ is in fact minimal in the case of $bu$.
Possibly it is reasonable to expect that the filtration constructed
in~\cite{A-L} is minimal in other cases as well and is strictly
smaller than the corresponding stable rank filtration. We also wonder if one can
learn something interesting in the case of the algebraic
$K$-theory of a ring $R$ by studying the relationship between the
stable rank filtration of Rognes and the modified stable rank filtration 
that we define in this paper.

The organization of the rest of the paper is as follows. 
In Section~\ref{section: Mbar}, we construct the complex $\{\Mbar(k)\}$
and prove various results about it, in particular
Propositions~\ref{prop: null map} and~\ref{prop: Mbar projective}
that were mentioned in the introduction. We
also discuss the calculus evidence for 
Conjecture~\ref{conjecture: bu Whitehead reduced}.  
In Section~\ref{section: general construction}, we introduce the
formalism of augmented $\Gamma$-spaces and define the modified stable
rank filtration. We prove various technical results about
constructions involving augmented $\Gamma$-spaces, which we need in
the subsequent section.  
In Section~\ref{section: comparisons}, we prove 
Theorem~\ref{thm: pushout theorem}, which gives a general result about
the relationship of the modified stable rank filtration to the
construction in~\cite{A-L}. This implies that the complex
$\{\Mbar(k)\}$ is closely related to the modified stable rank
filtration for complex $K$-theory, $bu$. Finally, we prove 
Proposition~\ref{prop: equivalence of Rognes filtrations} to show
that in the case of complex $K$-theory, the modified stable rank 
filtration agrees with Rognes's original filtration, and we use this 
to prove Proposition~\ref{prop: Rognes connectivity conjecture}, which
confirms Rognes's connectivity conjecture for the case of $bu$.

\section{\protect The reduced complex $\protect\Mbar(k)$}
\label{section: Mbar}
As we mentioned in the introduction, Kuhn and Priddy's proof of the
Whitehead Conjecture does not directly attack the complex $\{L(k)\}$
to show that it is a projective resolution of $\HZ$. Rather, they
use a related auxiliary complex $\{M(k)\}$, which they prove is
a projective resolution of $\HZ/p$. One relevant
difference between the two complexes is that the spectra $L(k)$ are
most naturally thought of as Thom spectra, while the spectra $M(k)$
are most naturally thought of as stable summands of classifying
spaces. Thus techniques from the transfer are more easily applied to
the spectra $M(k)$ than to $L(k)$.

In this section, we construct a complex
$\{\Mbar(k)\}$ of spectra that appears to be the $bu$-analogue of
the complex $\{M(k)\}$ in the classical situation.  
In Section~\ref{sec: construction of Mbar}, we recall the relevant
definitions in Kuhn's homological algebra of spectra, we construct
the spectra $\Mbar(k)$, and we conjecture that the complex
$\{\Mbar(k)\}$ is exact.
In Section~\ref{sec: collections of subgroups}, we give technical
material on universal spaces of collections of subgroups, which
leads to two results about $\Mbar(k)$ 
(Propositions~\ref{prop: null map} and~\ref{prop: Mbar projective}) 
that are parallel to what is known about the classical $M(k)$ 
(Propositions~\ref{prop: null map classical} and~\ref{prop: M projective}).
In Section~\ref{section: calculus evidence}, we give evidence based on
the calculus of functors for the exactness of the complex
$\{\Mbar(k)\}$.

\subsection{\protect Construction of the complex $\protect\Mbar(k)$}
\label{sec: construction of Mbar}
\hfill\\

Our first order of business is to define a complex in the context of $bu$
that is an analogue of the projective resolution of 
$\HZ/p$ in~\eqref{eq: mod p Whitehead}.
We recall the relevant homological concepts from
\cite{KuhnSpacelike} and review the construction of the
auxiliary complex $\{M(k)\}$ used by Kuhn and Priddy. Then we follow the
same template with different ingredients to construct the complex
$\{\Mbar(k)\}$, and we set out the parallels that we know and
those that we conjecture between $\{M(k)\}$ and $\{\Mbar(k)\}$.

\begin{definition}[\cite{KuhnSpacelike}, Section~2]
\hfill
\begin{enumerate}
\item
  A spectrum is called ``free'' if it is the suspension spectrum of a
  space, and it is called ``projective'' if it is a wedge summand of a
  free spectrum.
\item 
  A ``chain complex (over $E_0$) of spectra'' is a sequence of spectra and
  maps between them,
\begin{equation}\label{eq: basic diagram}
\cdots\xrightarrow{\partial_3} X_2\xrightarrow{\partial_2}
   X_1\xrightarrow{\partial_1} X_0\xrightarrow{\partial_0} E_0,
\end{equation}
together with an extension to a diagram of the form
\[
\cdots\xrightarrow{} X_2 \xrightarrow{q_2} E_2  
 \xrightarrow{i_1} X_1 \xrightarrow{q_1} E_1 
 \xrightarrow{i_0} X_0 \xrightarrow{q_0} E_0,
\]
where each sequence $E_{n+1}\xrightarrow{i_{n}} X_{n} \xrightarrow{q_{n}} E_n$ 
is a cofiber sequence. (We will often specify
explicitly only the diagram~\eqref{eq: basic diagram}, with the
extended diagram being implicitly understood.)
\item
A cofiber sequence of spectra $X\rightarrow Y\xrightarrow{q} Z$ is
called ``short exact'' if the map $\Omega^\infty q$ has a homotopy section, 
i.e., if there exists a map  
$f: \Omega^\infty Z \rightarrow \Omega^\infty Y$ 
in the homotopy category such that the 
composed map $(\Omega^\infty q)\circ f$ is a weak homotopy equivalence. 
A chain complex is called ``exact'' if each of the cofiber
sequences $E_{n+1}\rightarrow X_n \rightarrow E_n$ is short exact. 
\end{enumerate}
\end{definition}

\begin{remark}
Elementary diagram chasing establishes a bijective
  correspondence (up to a suitable notion of homotopy equivalence)
  between chain complexes over $E_0$ and filtrations of the spectrum
  $E_0$, i.e., diagrams of spectra
\[
F_0\longrightarrow F_1 \longrightarrow F_2\longrightarrow 
         \cdots\longrightarrow E_0\simeq\hocolim F_n.
\]
(See~\cite{KuhnSpacelike}, Remark 2.2.)  To associate a chain complex
to a filtration of this form, set $X_n=\Sigma^{-n} F_n/F_{n-1}$ and
$E_n=\Sigma^{-n} E_0/ F_{n-1}$. For the other direction, suppose we
are given a chain complex as in~\eqref{eq: basic diagram}. This
structure provides maps $E_n\rightarrow \Sigma E_{n+1}$ for each
$n$, and hence maps $E_0 \longrightarrow \Sigma^{n+1} E_{n+1}$. Define
$F_n$ to be the homotopy fiber of this map.
\end{remark}

The following lemma establishes an alternative criterion for a 
chain complex to be exact.

\begin{lemma}  \label{lemma: contracting homotopy}
A chain complex of spectra
\[
\cdots\xrightarrow{\partial_{n+1}} X_{n}
      \xrightarrow{\ \partial_n\ } X_{n-1}
      \cdots
      \xrightarrow{\partial_{1}} X_{0}
      \xrightarrow{\partial_{0}} X_{-1}=E_{0}
\]
is exact if and only if for $n\geq -1$ there exist maps 
$\epsilon_{n}: \Omega^\infty X_n\rightarrow \Omega^\infty X_{n+1}$
in the homotopy category such that for all $n\geq -1$ the map
\begin{equation}\label{eq: contracting homotopy}
\epsilon_{n-1}\circ \left(\Omega^\infty\partial_n \right)
        +\left(\Omega^\infty\partial_{n+1} \right)\circ \epsilon_n
\end{equation} 
is a weak homotopy equivalence. 
\end{lemma}

The maps $\epsilon_n$ will be called a ``contracting homotopy'' for
the chain complex.  Note that if a chain complex is exact, then the
homotopy spectral sequence for the associated filtration collapses 
at the second page.

\begin{proof}[Proof of Lemma~\ref{lemma: contracting homotopy}]
Suppose first that the chain complex is exact. Thus for each $n$ 
there is a map 
\[
f_{n}:\Omega^\infty E_n\longrightarrow \Omega^\infty X_n
\] 
that is a homotopy section of the structure map 
\[
\Omega^\infty q_{n}:\Omega^\infty X_{n}\longrightarrow \Omega^\infty E_n.
\] 
For each $n$, we have the following homotopy commutative diagram, 
where the rows are fibration sequences:
\[
\begin{CD} 
\Omega^\infty E_{n+1}@>>>\Omega^\infty E_{n+1}\times \Omega^\infty E_{n} 
       @>{\rm{proj}}>> \Omega^\infty E_n \\
@V=VV @VV{\Omega^\infty i_{n} + f_{n}}V @VV=V \\
\Omega^\infty E_{n+1} @>{\Omega^{\infty}i_{n}}>> \Omega^\infty X_n 
                      @>{\Omega^{\infty}q_{n}}>> \Omega^\infty E_n.
\end{CD}
\]
The middle vertical map must be an equivalence, and we can define
the map $\epsilon_{n}$ to be the composite 
\[
\Omega^\infty X_n 
   \xrightarrow{(\Omega^\infty i_n + f_n)^{-1}}
   \Omega^\infty E_{n+1}\times \Omega^\infty E_{n} 
   \xrightarrow{\rm{proj}} \Omega^\infty E_{n+1} 
   \xrightarrow{f_{n+1}}\Omega^\infty X_{n+1}.
\]
Elementary diagram chasing establishes that the maps $\epsilon_n$ 
satisfy~\eqref{eq: contracting homotopy}.

Conversely, suppose that we have a chain complex endowed with maps
$\epsilon_n$ satisfying~\eqref{eq: contracting homotopy}. We define
maps $f_n:\Omega^\infty E_n \rightarrow \Omega^\infty X_n$ by
$f_n=\epsilon_{n-1}\circ (\Omega^\infty i_{n-1})$.  An inductive
argument shows that $f_n$ is a homotopy section of 
$\Omega^\infty q_n$.
\end{proof}

Having established the terminology and notation for chain complexes of
spectra, we now turn to the auxiliary spectra $M(k)$ used by Kuhn and
Priddy.  In studying the filtration of $\HZ$ by $\Sp^{m}(\Sphere)$ at
a particular prime $p$, one can focus on values of $m$ that are powers
of $p$, because otherwise
$\Sp^{m-1}(\Sphere)\rightarrow\Sp^{m}(\Sphere)$ is an equivalence at
the prime $p$.
In point of fact, Kuhn and Priddy work with a version of the symmetric
power filtration that is reduced mod $p$, as follows. 
Let $D(k)$ be the cofiber of the $p$-fold diagonal map 
$\Sp^{p^{k-1}}(\Sphere)\rightarrow\Sp^{p^{k}}(\Sphere)$. 
The filtration of $\HZ$ by symmetric powers of the sphere spectrum
fits into a diagram with a filtration of $\HZ/p$ by the spectra
$D(k)$ as follows:
\begin{equation}    \label{eq: symm power setup}
\begin{CD}
*=\Sp^{0}(\Sphere) @>>> \Sp^{1}(\Sphere) @>>> \Sp^{p}(\Sphere) 
                    @>>> \ \dots\   @>>> \HZ \\
@VVV        @VV{\Delta}V         @VV{\Delta}V 
                          @.         @VV{\times p}V\\
\Sphere=\Sp^{1}(\Sphere) @>>> \Sp^{p}(\Sphere) @>>> \Sp^{p^{2}}(\Sphere) 
                    @>>> \ \dots\ @>>> \HZ \\
@VVV        @VVV       @VVV              @.     @VVV\\
\Sphere=:D(0) @>>> D(1)  @>>> D(2) 
                    @>>>\ \dots\ @>>> \HZ/p.
\end{CD}
\end{equation}
Here the horizontal maps between symmetric powers are induced by
basepoint inclusions, while the vertical maps are $p$-fold diagonal
maps, and all spectra are implicitly localized at $p$. The columns
are by definition homotopy cofiber sequences. To discuss the 
associated chain complexes, define
$L(k):=\Sigma^{-k}\Sp^{p^{k}}(\Sphere)/\Sp^{p^{k}-1}(\Sphere)$
and $M(k):=\Sigma^{-k}D(k)/D(k-1)$.
(At the beginning of the sequence, $L(0)=M(0)=\Sphere$.)

\begin{remark}    \label{remark: ambiguity}
  Recall once again that for $k>0$ the maps
  $\Sp^{p^{k-1}}(\Sphere)\rightarrow \Sp^{p^k-1}(\Sphere)$
  are equivalences at $p$.  Hence
  the spectra $\Sp^{p^k}(\Sphere)/\Sp^{p^{k-1}}(\Sphere)$ and
  $\Sp^{p^k}(\Sphere)/\Sp^{p^k-1}(\Sphere)$ are equivalent after localization
  at $p$, and either one could be used to define $L(k)$. In fact, for
  $k>0$, the difference between the two spectra is that the latter is
  $p$-local, while the former is not, so the latter is the
  $p$-localization of the former.
\end{remark}

Each row in \eqref{eq: symm power setup} now gives us a chain complex,
\begin{equation}    \label{eq: symm power chain complex}
\begin{CD}
\dots @>>> \Sigma^{-1}L(1) @>>> \Sigma^{-1}L(0) @>>> * @>>> \HZ\\
@. @VVV        @VVV         @VVV  @VV{\times p}V\\
\dots @>>> L(2) @>>> L(1) @>>> L(0) @>>> \HZ \\
@. @VVV        @VVV       @VVV             @VVV\\
\dots @>>> M(2) @>>> M(1) @>>> M(0) @>>> \HZ/p,
\end{CD}
\end{equation}
and again each column is a homotopy cofiber sequence. In fact, 
it turns out that this cofiber sequence splits. 

\begin{proposition}\cite[Proposition 5.15]{Mitchell-Priddy}
  \label{prop: null map classical} 
$M(k)$ splits as a wedge sum:
\[
M(k)\simeq L(k)\vee L(k-1).
\]
\end{proposition}

Another important feature of the chain complexes of 
diagram~\eqref{eq: symm power chain complex} is that each spectrum
$M(k)$ is projective, again from
Mitchell and Priddy's earlier work.  This is important in Kuhn and
Priddy's proof because they use the standard homological technique of
inducing a map from a projective complex to an acyclic complex.
Let $\Delta_{k}$ be a transitive elementary abelian $p$-subgroup 
of $\Sigma_{p^{k}}$. ($\Delta_{k}$ is unique up to conjugacy.) 

\begin{proposition}\cite[Theorem 5.1]{Mitchell-Priddy}
    \label{prop: M projective}
  $M(k)$ splits off from
  $\Sigma^{\infty}\left(B\Delta_{k}\right)_{+}$ as the stable wedge
  summand corresponding to the Steinberg idempotent 
  in $\field_{p}\left[GL_{k}(\field_{p})\right]$. 
\end{proposition}
%


To define the $bu$-analogue of this setup, we follow the same template
with different ingredients.  Recall that the construction
of~\cite{A-L}, applied to finite pointed sets and taken at powers of
$p$, gives exactly the filtration
\[
\Sphere=\Sp^{1}(\Sphere)
    \rightarrow \Sp^{p}(\Sphere)\rightarrow\Sp^{p^{2}}(\Sphere)
    \rightarrow\dots\rightarrow \HZ,
\]
while applying it to finite-dimensional complex vector spaces gives a
filtration
\[
bu=A_{0}\rightarrow A_{1} \rightarrow A_{p} 
    \rightarrow\dots\rightarrow \HZ.
\]
The $bu$-analogue of \eqref{eq: symm power setup} retains the top row,
and replaces the middle row by the spectra $A_{p^k}$. The map between
the first row and the second row is induced by the natural map of
categories from finite pointed sets to finite-dimensional complex
vector spaces, and we define spectra $C(k)$ as the vertical cofibers,
indexed logarithmically, by analogy\footnote{Note that since the
  sequence $\{C(k)\}$ is indexed logarithmically and the sequence
  $\{A_{p^{k}}\}$ is not, we have an initial term $C(-\infty)$,
  corresponding to $\log_{p}0=-\infty$.}  to the spectra $D(k)$
of~\eqref{eq: symm power setup}:
\begin{equation}    \label{eq: bu setup}
\begin{CD}
*        @>>> \Sp^{1}(\Sphere) @>>> \Sp^{p}(\Sphere) @>>> \dots @>>> \HZ\\
@VVV           @VVV                      @VVV             @.         @VVV\\
bu=A_{0} @>>> A_{1}            @>>> A_{p}          @>>> \dots @>>> \HZ\\
@VVV           @VVV                      @VVV             @.         @VVV\\
bu=:C(-\infty) @>>> C(0)             @>>> C(1)           @>>> \dots @>>> *
\end{CD}
\end{equation}

We now pass to subquotients with the following notation:
\begin{align*}
L(k)&:=\Sigma^{-k}\Sp^{p^{k}}(\Sphere)/\Sp^{p^k-1}(\Sphere)\\
T(k)&:=\Sigma^{-(k+1)}A_{p^{k}}/A_{p^k-1}\\
\Mbar(k) &:=\Sigma^{-(k+1)}C(k)/C(k-1).
\end{align*}
(Note that $A_{p^{k-1}}\rightarrow A_{p^k-1}$ is an equivalence at $p$,
and the appropriate variation of Remark~\ref{remark: ambiguity} applies to 
the definition of $T(k)$.)
Thus we obtain the following diagram of chain
complexes, where all spectra are implicitly localized at $p$, and in
fact, all spectra except for the rightmost three columns were
$p$-local from the outset.
\begin{equation}    \label{eq: bu chain complexes}
\begin{CD}
\dots @. \Sigma^{-1}L(2)@>>> \Sigma^{-1}L(1) @>>> \Sigma^{-1}L(0) 
     @>>> * @>>> \HZ\\
@.    @VVV    @VVV    @VVV         @VVV  @VV{\simeq}V\\
\dots @. T(2) @>>> T(1) @>>> T(0) @>>> bu @>>> \HZ \\
@. @VVV @VVV       @VVV       @VVV             @VVV\\
\dots @. \Mbar(2) @>>> \Mbar(1) @>>> \Mbar(0) @>>> bu @>>> *
\end{CD}
\end{equation}
(At the beginning of the sequence, $L(0)=\Sphere$,
$ T(0)=\Sigma^{\infty}\complexes P^{\infty}$, and 
%
%
$\Mbar(0)
    \simeq \Sigma^{\infty}\complexes P^{\infty}_{+}
    \simeq T(0)\vee L(0)$.)
The conjecture that corresponds to 
the exactness of \eqref{eq: mod p Whitehead}, which was the main 
computation performed in \cite{Kuhn-Priddy}, is the following. 

\begin{conjecture}[Reduced $bu$ Whitehead Conjecture]
\label{conjecture: bu Whitehead reduced}
\buWhiteheadConjectureText
\end{conjecture}

In the next two subsections, we elaborate further on this conjecture
in two ways. In Section~\ref{sec: collections of subgroups}, 
we establish analogues of Propositions~\ref{prop: null map classical} 
and~\ref{prop: M projective}. These establish that the spectra
in Conjecture~\ref{conjecture: bu Whitehead reduced} are projective
in the sense of Kuhn, and that the structure of 
diagram~\eqref{eq: bu chain complexes} is exactly parallel to that
of diagram~\eqref{eq: symm power setup}. 
Then in Section~\ref{section: calculus evidence}, we offer 
evidence from the calculus of functors for the correctness
of Conjecture~\ref{conjecture: bu Whitehead reduced}. 

\subsection{\protect Collections of subgroups} 
\label{sec: collections of subgroups}

In this subsection, we explore further the properties of the spectra
$\Mbar(k)$ that establish them as appropriate analogues of the spectra
$M(k)$ used by Kuhn and Priddy in the proof of the mod~$p$ Whitehead
Conjecture. First, comparing \eqref{eq: symm power chain complex} to 
\eqref{eq: bu chain complexes} and looking at the columns, we see
that the spectra $L(k-1)$, $L(k)$, and $M(k)$ in the classical
situation correspond, respectively, to $L(k)$, $T(k)$, and $\Mbar(k)$
in the $bu$-analogue. Thus the result that corresponds to 
Proposition~\ref{prop: null map classical} 
is the following. 

\begin{proposition}     \label{prop: null map}
\NullMapPropositionText{\eqref{eq: bu chain complexes}}
\end{proposition}

Second, we show that $\Mbar(k)$ is projective,
corresponding to Proposition~\ref{prop: M projective}. In fact, it
turns out that to describe how $\Mbar(k)$ splits off of a suspension
spectrum, we can use the dictionary that is used in \cite{A-L}
between the identity functor and symmetric groups on the one hand,
and the functor $V\mapsto BU(V)$ and the unitary groups on the other.
As explained in \cite{A-L} Section 10, the analogue of the
transitive elementary abelian $p$-subgroup $\Delta_{k}$ of
$\Sigma_{p^{k}}$ is the irreducible projective elementary abelian
$p$-subgroup $\Gamma_{k}$ of $U(p^{k})$. Just as the Weyl group of
$\Delta_{k}$ in $\Sigma_{p^{k}}$ is $GL_{k}(\field_{p})$, the Weyl
group of $\Gamma_{k}$ in $U(p^{k})$ is the symplectic group
$\Sp_{2k}(\field_{p})$.  Thus the statement that corresponds to
Proposition~\ref{prop: M projective} is the following.

\begin{proposition}     \label{prop: Mbar projective}
\MbarProjectivePropositionText
\end{proposition}

The objects in Propositions~\ref{prop: null map}
and~\ref{prop: Mbar projective} are defined in terms of subquotients
of filtrations obtained by the constructions of~\cite{A-L}. Since that
work identifies such subquotients in terms of collections of subgroups
of certain automorphism groups, we devote much of this subsection 
to related calculations with universal spaces.
The proofs of Propositions~\ref{prop: null map}
and~\ref{prop: Mbar projective} appear towards the end of the
subsection. 

We begin by recalling some essential background. (A good reference is 
\cite{A-D} Section~2.) Given a group $G$, a
``collection'' $\Ccal$ of subgroups of $G$ is a set of subgroups of
$G$ that is closed under conjugation by elements of $G$.
There is a universal $G$-space
$E\Ccal$, terminal among $G$-spaces whose isotropy groups are in
$\Ccal$, which can be characterized by the following two properties:
{\emph{(i)}} all isotropy groups of $E\Ccal$ are in $\Ccal$, and
{\emph{(ii)}} 
if $H\in\Ccal$, then $(E\Ccal)^{H}\simeq *$. The space
   $E\Ccal$ has a standard construction as the nerve of the category
   whose objects are pairs $(O,x)$, where $O$ is a transitive $G$-set
   with isotropy in $\Ccal$, and $x\in O$.  If we think of $\Ccal$ as
   a poset under inclusions and write $|\Ccal|$ for the nerve of the
   corresponding category, then there is a $G$-equivariant map
   $E\Ccal\rightarrow |\Ccal|$ given by taking $(O,x)$ to the isotropy
   group of $x$, and by \cite{A-D}~2.12 that map is $G$-equivariant 
   and a homotopy equivalence (but not necessarily a $G$-equivalence).

We are interested in particular collections of subgroups
of the unitary group $U(m)$. These collections were studied in 
\cite{A-L} Section~9. 
\begin{definition}   \label{defn: subgroup types}
Let $H$ be a subgroup of $U(m)$.
\begin{enumerate}
\item We say $H$ is ``standard''
if $H$ is conjugate to a proper subgroup of the form
$\prod_{i=1}^{s} U(m_{i})$, where $U(m_{i})$ is the group of automorphisms 
of a subspace $\complexes^{m_i}\subset \complexes^m$. \\
---The collection of standard subgroups of $U(m)$ is denoted $\Rcal_{m}$.
\item We say that a standard subgroup $H$ is ``complete'' if it is a proper 
subgroup conjugate to one of the form
$\prod_{i=1}^{s} U(m_{i})$ where $\sum m_{i}=m$.\\
---The collection of complete subgroups of $U(m)$ is denoted $\Lcal_{m}$.
We observe that $\Lcal_{m}$ is equivalent to the poset
of proper direct sum decompositions of $\complexes^{m}$. 
\end{enumerate}
\end{definition}
Note that all complete subgroups are standard; $U(m)$ itself is neither
standard nor complete; the trivial subgroup is standard but is not complete.

The heart of the proofs of 
Propositions~\ref{prop: null map} and~\ref{prop: Mbar projective}
is a commutative ladder involving both homotopy orbit spaces and 
strict orbit spaces of universal spaces of collections of subgroups. 
Let $X^{\diamond}$ denote the unreduced suspension of a space $X$ and,
if $X$ is a $G$-space with a based action, let 
$X_{\hobased G}$ be the based (reduced) homotopy orbits, 
$X_{\hobased G}:= EG_{+}\wedge_{G}X$. 
Let $\Rcal_{m,\nontrivial}$ be the subcollection of nontrivial standard
subgroups of $U(m)$.  By using the map $S^{0}\rightarrow S^{2m}$ that
includes $S^{0}$ as the poles of $S^{2m}$, and the passage from homotopy
orbits to actual orbits, we can construct the following commutative ladder:
\begin{equation}  \label{diagram: ladder}
\begin{CD}
(E\Rcal_{m,\nontrivial}\mbox{}^{\diamond}\wedge S^{0})_{\hobased U(m)}
   @>>> (E\Rcal_{m}\mbox{}^{\diamond}\wedge S^{0})_{\hobased U(m)}\\
@VVV @VVV\\
(E\Rcal_{m,\nontrivial}\mbox{}^{\diamond}\wedge S^{0})_{U(m)}
   @>>> (E\Rcal_{m}\mbox{}^{\diamond}\wedge S^{0})_{U(m)}\\
@VVV @VVV\\
(E\Rcal_{m,\nontrivial}\mbox{}^{\diamond}\wedge S^{2m})_{U(m)}
   @>>> (E\Rcal_{m}\mbox{}^{\diamond}\wedge S^{2m})_{U(m)}
\end{CD}
\end{equation}
When $m$ is a power of $p$, say $m=p^{k}$, applying $\Sigma^{\infty}$
to the middle row of this ladder turns out to give the $(k+1)$-fold
suspension of the map $\Sigma^{-1}L(k)\rightarrow T(k)$ in~\eqref{eq:
  bu chain complexes}, so the following lemma is an essential
ingredient in Proposition~\ref{prop: null map}. 

\begin{lemma}     \label{lemma: inclusion null}
$(E\Rcal_{m,\nontrivial}\mbox{}^{\diamond}\wedge S^{0})_{U(m)}
   \rightarrow (E\Rcal_{m}\mbox{}^{\diamond}\wedge S^{0})_{U(m)}$
is nullhomotopic.
\end{lemma}

\begin{proof}
Consider the lower square of~\eqref{diagram: ladder}.
By \cite{A-L} Proposition~9.13, we know that 
$(E\Rcal_{m,\nontrivial}\mbox{}^{\diamond}\wedge S^{2m})_{U(m)}\simeq *$. 
On the other hand, equation (9.3) in \cite{A-L} says that
\[
(E\Rcal_{m}\mbox{}^{\diamond}\wedge S^{0}) \rightarrow
(E\Rcal_{m}\mbox{}^{\diamond}\wedge S^{2m})
\]
is a $U(m)$-equivalence, and so
\[
(E\Rcal_{m}\mbox{}^{\diamond}\wedge S^{0})_{U(m)} \rightarrow
(E\Rcal_{m}\mbox{}^{\diamond}\wedge S^{2m})_{U(m)}
\]
is an equivalence, which completes the proof. 
\end{proof}

The remaining ingredients for the proof of 
Proposition~\ref{prop: null map} come from the upper square
of~\eqref{diagram: ladder}. The first step is to show that the upper
square is a homotopy pushout square and that it has a contractible
space in the upper right corner. The second step
is to identify the suspension spectra of the remaining corners of the
square with the appropriate suspensions of $L(k)$, $T(k)$, and
$\Mbar(k)$.

\begin{lemma}     \label{lemma: R cofiber seq}
The top square of \eqref{diagram: ladder} is a homotopy 
pushout square with contractible upper right corner. 
\end{lemma}

\begin{proof}
  The proof is essentially the same as that of Proposition~9.11
  of~\cite{A-L}. The map 
  $E\Rcal_{m,\nontrivial}\rightarrow E{\Rcal_{m}}$ is the inclusion of
  the singular set, as is easily checked by looking at the
  isotropy groups of the chains that form the simplices of
  $E{\Rcal_{m}}$. The cofiber of this map has a free action of $U(m)$
  (in the based sense), and so horizontal cofibers of the first two
  rows of of \eqref{diagram: ladder} are equivalent.  Because all of
  the spaces involved are simply connected (they are suspensions of
  connected spaces), this is sufficient to guarantee that the top
  square of \eqref{diagram: ladder} is a homotopy pushout square.
  Lastly, because $\Rcal_{m}$ contains the trivial subgroup, it
  follows that $E\Rcal_{m}\simeq |\Rcal_{m}|\simeq *$, and so
  $(E\Rcal_{m}\mbox{}^{\diamond}\wedge S^{0})_{\hobased U(m)}$ is
  contractible.
\end{proof}

To begin the identification of the suspension spectra of the corners
of the upper square of~\eqref{diagram: ladder}, note that 
the middle row is just the inclusion map 
$B\Rcal_{m,\nontrivial}\mbox{}^{\diamond} \rightarrow
     B\Rcal_{m}\mbox{}^{\diamond}$.
We recall from
\cite{A-L} Proposition~9.8 that the inclusion of the collection of
complete subgroups into standard subgroups induces a map
$E\Lcal_{m}\rightarrow E\Rcal_{m,\nontrivial}$ that is
a $U(m)$-equivalence, and therefore is still a homotopy equivalence
once we take orbit spaces. This allows us to replace the middle
row of~\eqref{diagram: ladder} with the map 
$B\Lcal_{m}\mbox{}^{\diamond} \rightarrow B\Rcal_{m}^{\diamond}$
and the top left corner with 
$\left(E\Lcal_{m}^{\diamond}\right)_{\hobased U(m)}$. 
On the other hand, the $U(m)$-equivariant map 
$E\Lcal_{m}\rightarrow |\Lcal_{m}|$ is a homotopy equivalence 
(though not a $U(m)$-equivalence), and so induces an equivalence
on homotopy orbits. Thus Lemmas~\ref{lemma: inclusion null}
and~\ref{lemma: R cofiber seq} tell us that the upper square 
of~\eqref{diagram: ladder} is equivalent to a
homotopy pushout square
\begin{equation}  \label{diagram: spectrum pushout square}
\begin{CD}
|\Lcal_{m}|^{\diamond}_{\hobased U(m)} @>>> *\\
@VVV @VVV\\
\left(B\Lcal_{m}\mbox{}^{\diamond}\right)
     @>{*}>> \left(B\Rcal_{m}\mbox{}^{\diamond}\right).
\end{CD}
\end{equation}

To prove Propositions~\ref{prop: null map} and~\ref{prop: Mbar projective}, 
we need to relate \eqref{diagram: spectrum pushout square} to 
the successive quotients of the spectra $A_{m}$ and $\Sp^{m}(\Sphere)$
for $m=p^{k}$. The most immediate is the lower right corner. 
It follows from \cite{A-L}~Corollary~8.3 that 
\begin{align}\label{eq: T(k)}
\Sigma^\infty(B\Rcal_{p^{k}}\mbox{}^{\diamond})
     &\simeq A_{p^{k}}/A_{p^{k}-1}
\intertext{We want to relate the lower left corner
of \eqref{diagram: spectrum pushout square} to the 
quotients of symmetric powers of spheres, which can also be identified
in terms of classifying spaces of collections of subgroups, this time
of symmetric groups rather than unitary groups. 
Let $\Fcal_{m}$ be the collection of proper standard subgroups of
$\Sigma_{m}$, i.e., subgroups of $\Sigma_{m}$ that are conjugate to 
subgroups of the form $\Sigma_{m_{1}}\times\dots\times\Sigma_{m_{s}}$. 
Then 
}
 \label{eq: L(k)}
     \Sigma^{\infty}\left(B\Fcal_{p^{k}}\mbox{}^{\diamond}\right)
     &\simeq Sp^{p^{k}}(\Sphere)/Sp^{p^{k}-1}(\Sphere).
\end{align}
The following proposition relates this to the lower left corner
of \eqref{diagram: spectrum pushout square}
through the inclusion $\Sigma_{m}\rightarrow U(m)$ that permutes the
standard basis elements of~$\complexes^{m}$. 

\begin{proposition}  \label{prop: BL is BF}
The $\Sigma_{m}$-equivariant inclusion $\Fcal_{m}\rightarrow \Lcal_{m}$ 
induces a homotopy equivalence $B\Fcal_{m}\simeq B\Lcal_{m}$.
\end{proposition}

To prove the proposition, one shows that the orbit categories whose 
nerves give $B\Fcal_{m}$ and $B\Lcal_{m}$ have isomorphic subcategories
containing at least one object from each isomorphism class. The essential
point is that standard subgroups of $\Sigma_{m}$ are complete, and
complete subgroups of $U(m)$ and of $\Sigma_{m}$ have isomorphic Weyl groups.

\begin{proof}[Proof of Proposition~\ref{prop: null map}]

We have a commuting diagram
\[
  \xymatrix{
\Sigma^{\infty}\left(B\Fcal_{p^{k}}\mbox{}^{\diamond}\right)
     \ar[r]^{\simeq} 
     \ar[dd]^{\simeq}  
&  \Sigma^{\infty}\left(B\Lcal_{p^{k}}\mbox{}^{\diamond}\right)
     \ar[r]^{\simeq} 
&  \Sigma^{\infty}\left(B\Rcal_{p^{k},\nontrivial}\mbox{}^{\diamond}\right)
     \ar[d]^{*}  
\\
& &
\Sigma^{\infty}\left(B\Rcal_{p^{k}}\mbox{}^{\diamond}\right)
     \ar[d]^{\simeq}  
\\
\Sp^{p^{k}}(\Sphere) /\Sp^{p^{k}-1}(\Sphere) 
      \ar[rr]
&
&
A_{p^{k}}/A_{p^{k}-1}      
}
\]
that results from combining equations \eqref{eq: T(k)}
and \eqref{eq: L(k)} with the results of Proposition~\ref{prop: BL is BF},
Lemma~\ref{lemma: inclusion null}, and the naturality of the
identification of filtration quotients in terms of classifying spaces
of collections of subgroups.  The $(k+1)$-fold desuspension of
$\Sp^{p^{k}}(\Sphere)/\Sp^{p^{k}-1}(\Sphere) \rightarrow
A_{p^{k}}/A_{p^{k}-1}$ is
$\Sigma^{-1}L(k)\rightarrow T(k)$, so
this finishes the proof of the proposition.
\end{proof}

\begin{proof}[Proof of Proposition~\ref{prop: Mbar projective}]
From \eqref{diagram: spectrum pushout square}, \eqref{eq: T(k)},
and \eqref{eq: L(k)}, we know that there is a homotopy cofiber
sequence
\[
\left(|\Lcal_{m}|^{\diamond}\right)_{\hobased U(m)} 
                \rightarrow \Sp^{m}(\Sphere)/\Sp^{m-1}(\Sphere)
                \rightarrow A_{m}/A_{m-1}.
\]
Since $\Mbar(k)$ is defined in Section~\ref{sec: construction of Mbar}
by
\[
\Mbar(k)\simeq \Sigma^{-(k+1)}
     \cof\left(A_{p^{k-1}}/\Sp^{p^{k-1}}(\Sphere)
            \rightarrow A_{p^{k}}/\Sp^{p^{k}}(\Sphere)\right),
\]
this establishes that
\begin{equation}
\label{eq: MBar is ho(Partition)}
\Mbar(k)\simeq
S^{-k} \wedge \left(|\Lcal_{p^{k}}|^{\diamond} \right)_{\hobased U(p^{k})}.
\end{equation}
On the other hand, by \cite{A-L} Proposition~10.3 and its proof, we know that 
the spectrum
$S^{-k}\wedge \left(|\Lcal_{p^{k}}|^{\diamond} \right)_{\hobased U(p^{k})}$
is a wedge summand of 
$\Sigma^{\infty}\left(B\Gamma_{k}\mbox{}^{\diamond}\right)_{+}$
by the symplectic Steinberg idempotent, which finishes the proof. 
\end{proof}

\subsection{Calculus evidence for variants of the Whitehead conjecture}
\label{section: calculus evidence}

In this subsection, we present evidence from the calculus of functors
for Conjecture~\ref{conjecture: bu Whitehead reduced}, the ``reduced
$bu$-Whitehead conjecture,'' as well as for the classical ``mod $p$ Whitehead
Conjecture'' when $p=2$. This follows on from the discussion
in~\cite{A-L} Sections 1 and~11, where we discussed a tantalizing link
between, on the one hand, the chain complexes of spectra
\begin{align}
\label{eq: old}
\dots \longrightarrow L(1)\longrightarrow L(0)\longrightarrow \HZ \\
\label{eq: new}
\dots \longrightarrow T(1)\longrightarrow T(0)
               \longrightarrow bu\longrightarrow \HZ 
\end{align}
and, on the other hand, certain ``Taylor towers'' arising from the
calculus of functors of Goodwillie and Weiss. For~\eqref{eq: old}, the
link is that $\Omega^{k-1}\Omega^\infty L(k)$ is equivalent to the
$k$-th nontrivial layer of the Goodwillie tower at the prime $p$ of
the identity functor evaluated at $S^1$. Similarly for~\eqref{eq:
  new}, $\Omega^{k-1}\Omega^\infty T(k)$ is equivalent to the $k$-th
nontrivial layer of the Weiss tower of the functor $V\mapsto BU(V)$ at
the prime $p$, evaluated at $\complexes$. This suggests that there may
be a deeper connection between the complexes~\eqref{eq: old}
and~\eqref{eq: new} and the Taylor towers.  In particular, the number
of loops involved to relate the infinite loop spaces of the spectra
in~\eqref{eq: old} and~\eqref{eq: new} to the layers of the Taylor
towers is just right; it allows the possibility that deloopings of the
structure maps in the Taylor towers may serve as a contracting
homotopy for the complexes, which would provide conceptual proofs of
exactness.  (See~\cite{ADL} for more on such deloopings.) We have not
been able to prove these speculations.

There exists a similar link to a Taylor tower for the reduced
$bu$-Whitehead Conjecture. 
In order to obtain a Taylor tower that potentially provides a
contracting homotopy for the complex
\[
\dots\longrightarrow \Mbar(2)\longrightarrow \Mbar(1)
                             \longrightarrow \Mbar(0) 
                             \longrightarrow bu
\]
of Conjecture~\ref{conjecture: bu Whitehead reduced}, 
we evaluate the Weiss tower for the functor $V\mapsto BU(V)$ 
at the vector space $\complexes^{0}$. We find (by
\cite{Arone-Topology}) that the fibers in the tower have the form
\[\Omega^{\infty}\map\left(|\Lcal_{p^k}|^{\diamond},
      \Sigma^{\infty}S^{\ad_{p^k}}\right)_{\hobased U(p^k)}.
\]
However, by Theorem~10.1 of \cite{A-L}, there is a mod $p$ equivalence
of spectra
\[
\map\left(|\Lcal_{p^k}|^\diamond, \Sigma^{\infty}S^{\ad_{p^k}}\right)
   \simeq_{p}
S^{-2k+1}\wedge|\Lcal_{p^k}|^\diamond.
\]
By \eqref{eq: MBar is ho(Partition)},
we know that $\Mbar(k)\simeq
S^{-k}\wedge\left(|\Lcal_{p^{k}}|^{\diamond} \right)_{\hobased U(p^{k})}$.
Thus we conclude that when the Weiss tower for 
$V\mapsto BU(V)$ is evaluated at $\complexes^{0}$,
there is a mod $p$ equivalence between the $k$-th layer and
the appropriate loop space of $\Mbar(k)$: 
\[
\Omega^{\infty}\map\left(|\Lcal_{p^k}|^{\diamond},
      \Sigma^{\infty}S^{\ad_{p^k}}\right)_{\hobased U(p^k)}
\simeq 
\Omega^{k-1}\Omega^{\infty}\Mbar(k).
\]
The connecting maps in the Taylor tower give us maps 
\begin{equation}   \label{eq: connecting map}
\Omega^{k-1}\Omega^{\infty}\Mbar(k)
\longrightarrow 
B\Omega^{k}\Omega^{\infty}\Mbar(k+1)
\simeq
\Omega^{k-1}\Omega^{\infty}\Mbar(k+1).
\end{equation}
Application of \cite{ADL} Corollary~7.2 shows that the connecting map
\eqref{eq: connecting map} deloops $k-1$ times, and so a delooping could
provide a contracting homotopy, as suggested in the following conjecture. 

\begin{conjecture}
(k-1)-fold deloopings of \eqref{eq: connecting map}
give a contracting homotopy for the complex of
Conjecture~\ref{conjecture: bu Whitehead reduced}.
\end{conjecture}

Finally, we record as a curiosity Bill Dwyer's observation that the
``mod $2$ Whitehead conjecture,'' which asserts the exactness of
the complex~\eqref{eq: mod p Whitehead} for the prime $2$, also
possesses a Taylor tower that might give a contracting homotopy.  
(A contracting homotopy was constructed in~\cite{Kuhn-Priddy} by other
methods.) However, we do not have a calculus analogue for the odd
primary analogue of this statement.  The point for $p=2$ is that $S^0$
can be identified with $\integers/2\simeq \Omega^\infty \HZ/2$. Thus,
the Taylor tower of the identity evaluated at $S^0$ gives rise to a
tower of fibrations converging to $ \Omega^\infty \HZ/2$, or at least
trying to converge.\footnote{It is not known if the Taylor tower for 
the identity converges at $S^{0}$. We speculate that it does. }
Proposition~\ref{prop: mod 2 layers} below, together with Theorem~1.17
and Corollary~9.6 of \cite{A-D}, tells us that the layers in this
Taylor tower coincide, up to the right number of deloopings, with the
terms $M(k)$ in~\eqref{eq: mod p Whitehead}. That is, 
Proposition~\ref{prop: mod 2 layers} tells us that
$\Omega^{k}\Omega^{\infty}M(k)$ is the $k$-th layer in the
$2$-localized Taylor tower for the identity evaluated at $S^{0}$.
Hence the connecting maps for this Taylor tower give maps
\[
\Omega^{k}\Omega^{\infty}M(k)\rightarrow \Omega^{k}\Omega^{\infty}M(k+1),
\]
and deloopings of these would provide potential contracting 
maps for \eqref{eq: mod p Whitehead}, as in the other situations we have
examined. 

\begin{proposition}    \label{prop: mod 2 layers}
At the prime $2$,
\[
M(k)\simeq_{2} S^{-(k-1)}\wedge \left(
                       |\Pcal_{2^{k}}|^{\diamond} 
                  \right)_{\hobased \Sigma_{2^{k}}}
\]
\end{proposition}

\begin{proof}
We know from Theorem~1.17 and Corollary~9.6 of \cite{A-D} 
(applied with $p=2$ and $X=S^{0}$) that 
$S^{-(k-1)}\wedge \left(
                       |\Pcal_{2^{k}}|^{\diamond} 
                  \right)_{\hobased \Sigma_{2^{k}}}$
is the Steinberg summand of $\Sigma^{\infty}\left(B\Delta_{2^{k}}\right)_{+}$.
On the other hand, Mitchell and Priddy \cite{Mitchell-Priddy}
proved that (at all primes) $M(k)$ is the Steinberg wedge summand of
$\Sigma^{\infty}\left(B\Delta_{p^{k}}\right)_{+}$, and this
concludes the proof. 
%
\end{proof}

\section{\protect General constructions}
\label{section: general construction}
In this section, we make technical preparations for 
Section~\ref{section: comparisons}, where our goal is to compare the filtration
constructed in~\cite{A-L} to Rognes's stable rank filtration of algebraic
$K$-theory~\cite{Rognes}. 
Rognes defined his stable rank filtration only for the case of the
algebraic $K$-theory of a discrete ring, but his construction can be
applied to topological $K$-theory as well, and the topological case is
covered by our discussion. The comparison in Section~\ref{section:
  comparisons} goes through an intermediary, the ``modified stable
rank filtration'' of a $K$-theory spectrum, which we introduce and
study in this section. It is based on Segal's $\Gamma$-space and
$\Gamma$-category constructions, as opposed to Waldhausen's $S_\bullet$
construction, which formed the setting of Rognes's original work. 
In other respects, however, it follows the spirit of~\cite{Rognes}.

In Section~\ref{section: filtered Gamma spaces}, we define the 
notion of a filtered $\Gamma$-space, we show how it arises from
an augmented permutative category, and we define the modified stable
rank filtration of a $K$-theory spectrum. In
Section~\ref{sec: sums and products}, we compare bar constructions
and homotopy pushouts of monoidal augmented $\Gamma$-spaces. 
Finally, in Section~\ref{section: very special Gamma spaces}
we study the particular situation of ``very special'' $\Gamma$-spaces, 
and we give a homotopy pushout diagram that allows an inductive 
understanding of the modified stable rank filtration for a 
permutative category.

\subsection{\protect Filtered and augmented $\protect\Gamma$-spaces}
\label{section: filtered Gamma spaces}

We begin by reviewing the relationship between $\Gamma$-spaces and
spectra and discussing how filtered $\Gamma$-spaces can serve as
convenient models for filtered spectra. We then consider how an
augmented permutative category gives rise to a filtered
$\Gamma$-space. We end this subsection with the definition of the
modified stable rank filtration of the spectrum associated to an
augmented permutative category, such as the algebraic $K$-theory
spectrum of a ring, and we give examples.

Let $\Gamma^{op}$ be the skeletal
category of finite pointed sets, having one object for each
cardinality. We will denote a generic object of $\Gamma^{op}$ by
$\n=\{0, 1, \ldots, n\}$, with zero acting as the basepoint. As
usual, a $\Gamma$-space $F$ is a pointed functor from $\Gamma^{op}$ to
the category of pointed simplicial sets.  A $\Gamma$-space may be
prolonged to a functor from the category of pointed simplicial sets to
itself via the diagonal of the levelwise evaluation of the functor on 
a simplicial set, once we replace it with its left Kan extension 
from the category of finite pointed sets to the category of all pointed sets.

A $\Gamma$-space comes equipped with a natural ``assembly map''
\[
X\wedge F(Y)\longrightarrow F(X\wedge Y),
\]
where $X$ and $Y$ are pointed simplicial sets, and so there
are suspension maps 
$S^{1}\wedge F(X)\longrightarrow F(S\sp{1}\wedge X)$. 
It follows that the sequence
\[
F(S^{0}), F(S^{1}), \ldots, F(S^{j}), \ldots 
\] 
forms a prespectrum, which we call the ``stabilization of $F$.''
A map $\alpha: F\rightarrow G$ between $\Gamma$-spaces is, by 
definition, a natural transformation between the underlying functors. 
The map $\alpha$ is called an ``equivalence'' if $\alpha(X)$ is a weak
equivalence for all $X$, and it is called a 
``stable equivalence'' if the associated map of stabilizations
\[
\{\alpha(S^{j})\}\colon \{F(S^{j})\} \longrightarrow \{G(S^{j})\}
\]
is a weak homotopy equivalence of spectra. The category of
$\Gamma$-spaces provides a good model for the category of
$(-1)$-connected spectra, via the functor $F\mapsto\{F(S^{j})\}$
\cite{B-F}, \cite{Lydakis-Gamma}.  

We are interested in filtrations of certain spectra, so we also 
set up a notion of ``filtered $\Gamma$-space'' as a model for
a filtered spectrum.

\begin{definition}
A ``filtered $\Gamma$-space'' is a sequence of $\Gamma$-spaces of the form
\[
F_{0}\longrightarrow F_{1}\longrightarrow\cdots\longrightarrow 
    F_{m}\longrightarrow \cdots\longrightarrow F=\underset{m}{\colim} F_{m}
\]
such that each of the maps $F_{m-1}\rightarrow F_{m}$ is an 
(objectwise) injection. 
\end{definition}

There is a self-evident notion of a map between filtered
$\Gamma$-spaces. We will denote a generic filtered $\Gamma$-space
sometimes by $F$, and sometimes by $\{F_{m}\}$.

\begin{definition}
  Let $F=\{F_{m}\}$ and $G=\{G_{m}\}$ be filtered $\Gamma$-spaces. Let
  $\alpha=\{\alpha_{m}\}$ be a filtered morphism from $F$ to $G$. We say
  that $\alpha$ is a ``filtered stable equivalence'' if $\alpha_{m}$
  is a stable equivalence for all $m$.
\end{definition}

The prototypical example of a filtered $\Gamma$-space is the infinite
symmetric product functor $\Sp^\infty$, filtered by the functors
$\Sp^{m}$ defined by $\Sp^{m}(X)=X^{m}/\Sigma_{m}$.  In this paper, we
are mostly concerned with filtrations that are pulled back from this
filtration of $\Sp^\infty$.

\begin{definition}   \label{defn: augmentation}
  An ``augmentation'' of a $\Gamma$-space $F$ is a map of
  $\Gamma$-spaces $\epsilon: F\rightarrow \Sp^{\infty}$ such
  that for all $X$, the inverse image of the basepoint of
  $\Sp^{\infty}(X)$ under the augmentation $\epsilon(X)$ consists of just
  the basepoint of $F(X)$. In this case $F$ will be called an ``augmented
  $\Gamma$-space.''
\end{definition}

We can use an augmentation to endow a $\Gamma$-space
$F$ with a filtration by $\Gamma$-spaces as follows.

\begin{definition}
Let $F$ be an augmented $\Gamma$-space. For $m=0, 1, 2,\ldots$, define
 $\Rcal_{m}F$ to be the strict limit (not the homotopy limit) 
of the following diagram:
\[
\begin{CD}
@. F\\
@. @VVV\\
\Sp^m @>>> \Sp^\infty .
\end{CD}
\]
\end{definition}

The $\Gamma$-spaces $\Rcal_{m}F$ endow $F$ with a filtration
\begin{equation}   \label{eq: filtration}
*=\Rcal_{0}F \longrightarrow \Rcal_{1}F
             \longrightarrow \cdots
             \longrightarrow F=\underset{m}{\colim}\Rcal_{m} F
\end{equation}
and we call it ``the filtration associated with the augmentation of
$F$.'' Because the composites 
$\Rcal_{m}F\rightarrow F\rightarrow\Sp^{\infty}$ provide
compatible augmentations for the $\Gamma$-spaces $\Rcal_{m}F$, 
the filtration $\{\Rcal_mF\}$ is actually a filtration of $F$ in the
category of augmented $\Gamma$-spaces. 

From here on we will work mostly in the category of augmented
$\Gamma$-spaces. 
\begin{definition}    \label{defn: augmented equivalence}
\hfill
\begin{enumerate}
\item 
Morphisms between augmented $\Gamma$-spaces that
respect the augmentations are called ``augmented morphisms.''
\item 
An augmented morphism $\alpha: F\rightarrow G$ between augmented
  $\Gamma$-spaces is called a ``strong augmented equivalence''
 (resp., ``strong augmented stable equivalence'') 
  if the induced map 
  $\Rcal_{m}\alpha:\Rcal_{m}F\longrightarrow \Rcal_{m}G$ is an equivalence 
  (resp. stable equivalence) for each $m$. 
\item 
  Two augmented $\Gamma$-spaces are said to be ``strongly (stably) equivalent''
  if they are related by a possibly zigzagging chain of strong augmented
  (stable) equivalences.
\end{enumerate}
\end{definition}


In order to define the modified stable rank filtration, we actually
want to filter augmented $\Gamma$-spaces that come from Segal's
construction of the $\Gamma$-space associated to a symmetric monoidal
category.  For technical convenience, we use a permutative category,
that is, a symmetric monoidal category in which the unit and
associativity isomorphisms are actually identity morphisms.  A basic
example is the ``category'' of nonnegative integers $\naturals$, where
the only morphisms are the identity morphisms, and the permutative
structure is given by addition.

Let $\Ccal$ be a category with sums such that the associated symmetric
monoidal category is in fact permutative.  (For example, take $\Ccal$
to have just one object in each isomorphism class.) Segal gave a
construction that associates to $\Ccal$ a $\Gamma$-category, i.e., a
functor $S\mapsto\Ccal(S)$ from $\Gamma^{op}$ to categories,
satisfying certain conditions.  The category $\Ccal(S)$ has objects
$(f,\alpha_{*})$ consisting of {\emph{(i)}} a function $f$ taking
  pointed subsets of $S$ to objects of $\Ccal$, and {\emph{(ii)}}
a collection of
  compatible isomorphisms $\alpha_{S_{1},S_{2}}: f(S_{1})\oplus
  f(S_{2})\xrightarrow{\cong} f(S_{1}\vee S_{2})$, one isomorphism for
  each pair of subsets $S_{1}$ and $S_{2}$ of $S$ whose intersection
  is the singleton set containing the basepoint.  Morphisms of such
  objects are required to be isomorphisms on the objects of $\Ccal$
  (or weak equivalences, in the context in which this is meaningful).
  There is a $\Gamma$-space associated to a $\Gamma$-category by
  taking nerves, and for simplicity, we do not distinguish between the
  two. That is, given a set $S$, we will write $\Ccal(S)$ for either
  the category $\Ccal(S)$ or its nerve, trusting to context to clarify
  which is meant. For example, the $\Gamma$-space associated with
  permutative category $\naturals$ is $\naturals(X)=\Sp^\infty(X)$.

\begin{definition}   \label{defn: augmented category}
  An ``augmentation'' of a permutative category $\Ccal$ is a 
  functor $\epsilon: \Ccal \rightarrow \naturals$ that respects the
  monoidal structure and has the property that that
  $\epsilon^{-1}(0)$ consists of just the zero object and its identity
  morphism. We write $\Ccal_{m}$ for the full subcategory of $\Ccal$
  given by $\epsilon^{-1}(m)$. 
\end{definition} 

If $\Ccal$ is an augmented permutative category, then the associated
$\Gamma$-space $\Ccal(X)$ is augmented (Definition~\ref{defn: augmentation}) 
because the augmentation $\epsilon: \Ccal\rightarrow\naturals$ induces
a natural transformation $\Ccal(X)\rightarrow\naturals(X)=\Sp^{\infty}(X)$.
Therefore, $\Ccal(X)$ is equipped with a natural filtration
as in~\eqref{eq: filtration}, pulled back from the filtration of
$\Sp^{\infty}(X)$.

\begin{definition}    \label{defn: filter C(X)}
For an augmented permutative category $\Ccal$, we denote the filtration
of $\Ccal(X)$ associated to the augmentation in the following way:
\[
\Rcal_0\left[\Ccal(X)\right] 
     \longrightarrow \Rcal_1\left[\Ccal(X)\right]
     \longrightarrow\cdots
     \longrightarrow \Rcal_m\left[\Ccal(X)\right] 
     \longrightarrow\cdots
     \longrightarrow \Rcal_\infty\left[\Ccal(X)\right].
\]
\end{definition}

Here, by definition we have $\Rcal_0\left[\Ccal(X)\right]=*$ and 
$\Rcal_\infty\left[\Ccal(X)\right]=\Ccal(X)$. Observe that 
$\Rcal_{m}/\Rcal_{m-1}\left[\Ccal(\one)\right] 
    \cong \left(B\Ccal_{m}\right)_{+}$.

Recall that given a permutative category $\Ccal$, the spectrum
\[
\Ccal\left(S^{0}\right), \Ccal\left(S^{1}\right), 
    \dots, \Ccal\left(S^{j}\right), \dots
\]
is called the (Segal) $K$-theory spectrum of the category $\Ccal$ and
is denoted $\kay\Ccal$.  
In fact, $\Ccal(-)$ is ``very special $\Gamma$-space,'' i.e., the
$n$ collapse maps $\n\rightarrow \one$ induce a homotopy
equivalence
\[
\Ccal(\n)\longrightarrow \Ccal(\one)^{n}. 
\]
As a consequence the spectrum $\{\Ccal(S^{j})\}$ is actually an
$\Omega$-spectrum for $j\geq 1$ \cite{Segal}.  


As soon as we filter $\Ccal(X)$ as in Definition~\ref{defn: filter C(X)}, 
however, the $\Gamma$-spaces $\Rcal_{m}[\Ccal(-)]$
are no longer ``very special.'' Nonetheless, the
mere fact that $\Rcal_{m}\left[\Ccal(-)\right]$ is a $\Gamma$-space 
gives us a suspension map 
\[
S^{1}\wedge \Rcal_{m}\left[\Ccal(S^{j}) \right]
      \longrightarrow \Rcal_{m}\left[\Ccal(S^{j+1})\right],
\]
and thus $\{\Rcal_{m}\left[\Ccal(S^{j})\right] \}_{j}$
is a spectrum, even if not an  $\Omega$-spectrum. 

\begin{definition}     \label{defn: modified stable rank filtration}
If $\Ccal$ is a permutative category, we define the spectrum 
$\Rcal_{m}\kay\Ccal:= \{\Rcal_{m}\left[\Ccal(S^{j})\right]  \}_{j}$
to be the ``$m$-th modified stable rank filtration'' of the $K$-theory
spectrum $\kay\Ccal$.  
\end{definition}

The terminology ``modified stable rank filtration" is justified in
Section~\ref{section: comparison modified to Rognes}, where we compare
this filtration to that of Rognes and show that they are closely
related in result as well as in construction. 

\begin{example}
\hfill
\begin{enumerate}
\item If $\Ccal=\naturals$, then by definition we have 
$\Rcal_m\left[\Ccal(X)\right]=\Sp^{m}(X)$, and so 
$\Rcal_{m}\HZ=\Rcal_{m}\kay\naturals=\Sp^{m}(\Sphere)$.
\item Let $\Ccal$ be the skeletal category of pointed finite pointed sets with
objects $\n$, morphisms given by isomorphisms, and permutative
structure given by wedge sum with, for each $n$ and $k$, 
a fixed choice of isomorphism of $\n\vee \kund$ with 
$\underline{n+k}$. It is well known that the associated spectrum 
$\kay\Ccal$ is the sphere spectrum $\Sphere$. 

The usual model for $\Ccal(X)$ is 
\[
\left(\coprod_{i\in\naturals} 
    E\Sigma_{i}\times_{\Sigma_{i}}X^{i}\right)
          \ / \sim
\]
where the identification $\sim$ identifies points in the $i$-th summand
to the $(i-1)$-st summand if they have coordinates at the basepoint of 
$X$. The augmentation to $\Sp^{\infty}(X)$ collapses each $E\Sigma_{i}$
to a point. We can see from this model that if $\Ccal$ is finite pointed
sets, then 
\[
\Rcal_{m}[\Ccal(X)]/\Rcal_{m-1}[\Ccal(X)]
\simeq 
\left(E\Sigma_{m}\right)_{+}\wedge_{\Sigma_{m}} X^{\wedge (m)}
\simeq 
\left(X^{\wedge (m)}\right)_{\hobased\Sigma_{m}}
\]
If $m>1$, then the quotient is out of the stable range, and 
the inclusion $\Rcal_{m-1}\Ccal(-)\rightarrow\Rcal_{m}\Ccal(-)$ 
is actually an augmented stable equivalence of $\Gamma$-spaces. 
The modified stable rank filtration is trivial in 
this case: 
$\Rcal_{0}\kay\Ccal=*$, $\Rcal_{1}\kay\Ccal=\kay\Ccal=\Sphere$, and
$\Rcal_{m}\kay\Ccal/\Rcal_{m-1}\kay\Ccal\simeq *$ for $m>1$. 
\end{enumerate}
\end{example}

\subsection{Sums, products, and pushouts of augmented $\protect\Gamma$-spaces}
\label{sec: sums and products}
In this subsection, we set up machinery to use in
Section~\ref{section: comparison of constructions} for the comparison of
the modified rank filtration constructed in 
Section~\ref{section: filtered Gamma spaces} with the 
filtration constructed in \cite{A-L}. 
The filtration of \cite{A-L} is based on an inductive
bar construction at the level of permutative categories and 
infinite loop spaces. However, we would like to make our comparison
in Section~\ref{section: comparison of constructions} using an inductive 
homotopy pushout construction in the category of augmented
$\Gamma$-spaces. 
Thus our main result of this section,
Proposition~\ref{proposition: two pushouts}, translates 
between the two, giving a strong augmented stable equivalence. 


Homotopy pushouts are based on sums, while bar constructions are 
based on products, so 
our first order of business is to compare sums and products.
Let $F$ and $G$ be $\Gamma$-spaces (not necessarily filtered or
augmented). It is a standard fact that the natural inclusion 
$F\vee G\longrightarrow F\times G$ is a stable equivalence, and we
need to extend this to filtered $\Gamma$-spaces. So let
$F=\{F_m\}$ and $G=\{G_m\}$ be filtered $\Gamma$-spaces.  We define
filtrations of the $\Gamma$-spaces $F\vee G$ and $F\times G$ as follows:
\begin{align} 
(F\vee G)_{m}   &=F_{m}\vee G_{m} 
            \label{eq: define filtration sum}\\
(F\times G)_{m} &=\underset{i+k\le m}{\colim} (F_{i}\times G_{k})
            \label{eq: define filtration product}
\end{align}
The inclusion $F\vee G\longrightarrow F\times G$ respects these
filtrations.

\begin{lemma}\label{lemma: filtered wedge and product}
  If $F$ and $G$ are filtered $\Gamma$-spaces, then the 
  inclusion map $F\vee G\rightarrow F\times G$ is a filtered stable
  equivalence. More generally, for any positive integer $l$ and filtered
  $\Gamma$-spaces $F^1,\ldots, F^l$, the inclusion map
$
F^1\vee\cdots\vee F^l\rightarrow F^1\times\cdots \times F^l
$
is a filtered stable equivalence.
\end{lemma}

\begin{proof}
We need to show that for each $m$, the inclusion map 
\[
F_{m}\vee G_{m} \longrightarrow \underset{i+k\le m}{\colim} F_{i}\times G_{k}
\]
is a stable equivalence. We consider the factorization
\begin{equation} \label{eq: factorization}
F_{m}\vee G_{m} 
 \longrightarrow
    \underset{\{i+k\le m\mid i\cdot k=0\}}{\colim} F_{i}\vee G_{k}
 \longrightarrow 
    \underset{i+k\le m}{\colim} F_{i}\vee G_{k}
 \longrightarrow 
    \underset{i+k\le m}{\colim} F_{i}\times G_{k},
\end{equation}
and we assert that each of these maps is a stable equivalence. 
In fact, it is easy to check that the first map is an isomorphism, and
the second map is likewise an isomorphism because the target
diagram is the left Kan extension of the source diagram. Finally,
consider the diagram
\[
\begin{CD}
\underset{i+k\le m}{\hocolim} F_{i}\vee G_{k} 
@>>>
\underset{i+k\le m}{\hocolim} F_{i}\times G_{k} 
\\
@VVV @VVV\\
\underset{i+k\le m}{\colim} F_{i}\vee G_{k}
@>>>
\underset{i+k\le m}{\colim} F_{i}\times G_{k}.
\end{CD}
\]
The vertical maps are equivalences, because both diagrams are
cofibrant~\cite{D-S}, and thus colimits are equivalent to homotopy colimits.
The top row is a stable equivalence because each $F_{i}\vee G_{k}
\rightarrow F_{i}\times G_{k}$ is a stable equivalence; thus the
bottom row is a stable equivalence, completing the proof
that~\eqref{eq: factorization} is a stable equivalence.

The more general statement follows by induction.
\end{proof}

If $F$ and $G$ are augmented $\Gamma$-spaces, then there are 
natural augmentations on the $\Gamma$-spaces $F\vee G$ and $F\times G$ 
defined by the compositions
\begin{align}
F\vee G&\longrightarrow \Sp^{\infty}\vee\Sp^{\infty} 
   \xrightarrow{\mbox{\tiny{ fold}}} \Sp^{\infty}
   \label{eq: sum augmentation}
\\
F\times G &\longrightarrow \Sp^\infty\times \Sp^\infty 
   \xrightarrow{\ +\ } \Sp^\infty.
   \label{eq: product augmentation}
\end{align}
The filtrations of $F\vee G$ and $F\times G$ that are induced
by these augmentations turn out to coincide with those
defined by \eqref{eq: define filtration sum} 
and~\eqref{eq: define filtration product}, as established in the
following lemma.

\begin{lemma}
If $F$ and $G$ are augmented $\Gamma$-spaces and $F\vee G$
and $F\times G$ are augmented by \eqref{eq: sum augmentation}
and~\eqref{eq: product augmentation}, then 
\begin{align*}
\Rcal_{m}(F\vee G)
    &\cong\left(\Rcal_{m}F\right)\vee\left(\Rcal_{m}G\right)\\
\Rcal_{m}(F\times G)
    &\cong\underset{i+k\le m}{\colim} \left(\Rcal_{i}F\times \Rcal_{k}G\right).
\end{align*}
\end{lemma}

\begin{proof}
The lemma follows for the product by inspection of the two 
stacked strict pullback diagrams
\[
\begin{CD}
\underset{i+k\le m}{\colim} \left(\Rcal_{i}F\times \Rcal_{k}G\right)
@>>> F\times G\\
@VVV @VVV\\
\underset{i+k\le m}{\colim} 
       \left(\Sp^{i}\times\Sp^{k}\right)
@>>> \Sp^{\infty}\times \Sp^{\infty}\\
@VV{\colim(+)}V @VV{+}V\\
\Sp^{m} @>>> \Sp^{\infty},
\end{CD}
\]
because the outer square is the pullback that identifies
$\Rcal_{m}(F\times G)$. The proof for the wedge is similar, but easier. 
\end{proof}

\begin{corollary}   \label{cor: augmented equivalence}
If $F$ and $G$ are augmented $\Gamma$-spaces and $F\vee G$
and $F\times G$ are augmented by \eqref{eq: sum augmentation}
and~\eqref{eq: product augmentation}, then the natural inclusion
$F\vee G \rightarrow F\times G$ is a strong augmented stable
equivalence. 
\end{corollary}


Finally, we come to our goal for this subsection, which is 
to compare bar constructions and homotopy pushouts of $\Gamma$-spaces. 
We call an augmented $\Gamma$-space $F$ an ``augmented
monoid'' if there is an associative and unital map of augmented
$\Gamma$-spaces $F\times F\rightarrow F$. 
Suppose given a diagram
\begin{equation}    \label{eq: pushout of monoids}
F_{1}\longleftarrow F_{0} \longrightarrow F_{2}
\end{equation}
of augmented monoids. We can define the augmented $\Gamma$-space
$\BAR(F_{1},F_{0},F_{2})$ as the diagonal of the standard simplicial object
\[
\q\mapsto F_{1}\times (F_{0})^q\times F_{2},
\]
where the augmentation is given levelwise by~\eqref{eq: product augmentation}.

On the other hand, we can define $F_{12}$ as the objectwise homotopy 
pushout of~\eqref{eq: pushout of monoids}, and $F_{12}$ has an augmentation
because there is a (strictly) commuting diagram
\[
\begin{CD}
F_{0}(X) @>>> F_{1}(X)\\
@VVV @VVV\\
F_{2}(X) @>>> \Sp^{\infty}(X)\\
\end{CD}
\]
A common canonical model for
$F_{12}$ is the geometric realization (i.e., diagonal) of the
simplicial object
\[
\q\mapsto F_{1}\vee \underbrace{F_{0}\vee\cdots\vee F_{0}}_{q}\vee F_{2}.
\]
With this model, we see that, levelwise, the augmentation of $F_{12}$ is 
just the wedge of the augmentations of $F_{0}$, $F_{1}$, and 
$F_{2}$. There is a natural augmentation-preserving map
$F_{12}\longrightarrow \BAR(F_{1}, F_{0}, F_{2})$ induced 
by the simplicial map that is given in degree $q$ by the inclusion map
\[
 F_{1}\vee (F_{0})^{\vee q}\vee F_{2} 
       \longrightarrow F_{1}\times (F_{0})^q\times F_{2}.
\]
We now have all of the ingredients for the main goal of this
subsection. 

\begin{proposition}\label{proposition: two pushouts}
\propositionpushouttext
\end{proposition}


\begin{proof}
By Corollary~\ref{cor: augmented equivalence}, 
$F_{12}\longrightarrow \BAR\left(F_{1}, F_{0}, F_{2}\right)$
is a strong augmented stable equivalence in each simplicial degree. It follows
that the induced map of geometric realizations is a strong augmented stable
equivalence.
\end{proof}

\subsection{\protect Very special $\Gamma$-spaces}
\label{section: very special Gamma spaces}

In this subsection, we study the structure of an augmented $\Gamma$-space
a little more closely.  We consider a map $F\rightarrow G$ of
augmented $\Gamma$-spaces that are ``very special'' in the sense of
Segal. We show that if $\Rcal_{i}F\rightarrow\Rcal_{i}G$ is an
equivalence for $i<m$, then we can compare $\Rcal_{m}F$
and $\Rcal_{m}G$ using the values of $F$ and $G$ on the
space $S^{0}$ (Proposition~\ref{proposition: equivalent cats}). Further, unlike
the previous subsection, this result does not require stabilization, but 
is true on the level of spaces rather than spectra. 
It generalizes the result of Theorem~1.1 of~\cite{Lesh-TAMS}
(Remark~\ref{remark: supersedes}).

The main tool for the comparison is a particular instance of the 
assembly map discussed in Section~\ref{section: filtered Gamma spaces}. 
Let $F$ be an augmented $\Gamma$-space and let $\kund$ denote the 
set $\{0,1,...,k\}$, with $0$ as the basepoint.
We can define a new augmented $\Gamma$-space $X\mapsto F(\one)\wedge X$
by using the assembly map to provide an augmentation via the composition
\[
F(\one)\wedge X\longrightarrow F(\one\wedge X)
                \xrightarrow{\ \cong\ }F(X)
                \longrightarrow\Sp^{\infty}(X). 
\]
By construction, this gives a map of augmented $\Gamma$-spaces 
$F(\one)\wedge X \longrightarrow F(X)$. 

\begin{example}  \label{example: assembly}
\hfill
\begin{enumerate}
\item
If $F=\Sp^\infty$, then $F(\one)=\naturals$, and the map
$\Sp^{\infty}(\one)\wedge X \longrightarrow \Sp^{\infty}(X)$ is given by
$m\wedge x\mapsto mx$. \label{item: sp}
\item
If $F$ is the $\Gamma$-space associated to the Segal construction for
category of finite pointed sets (so that for connected $X$ we have 
$F(X)\simeq \Omega^{\infty}\Sigma^{\infty}X$), then 
$F(\one)\simeq \bigvee_{i} \left(B\Sigma_{i}\right)_{+}$.
On each component of $F(\one)\wedge X$,
the map $F(\one)\wedge X \longrightarrow F(X)$ becomes the map
\[
B\Sigma_{i}\mbox{}_{+}\wedge X \rightarrow
     \left(\coprod E\Sigma_{i}\times_{\Sigma_{i}} X^{i}\right)/\sim
\]
induced by the $i$-fold diagonal map $X\rightarrow X^{i}$. 
\end{enumerate}
\end{example}

We need more detail on the filtration associated to the augmentation
of the $\Gamma$-space $X\mapsto F(\one)\wedge X$.
Example~\ref{example: assembly}\eqref{item: sp} shows that
$\Rcal_{m}\left[\Sp^{\infty}(\one)\wedge X\right] =
\Sp^{m}(\one)\wedge X $, and this turns out to be the key to the
general case, as shown in the proof of the following lemma.

\begin{lemma}   \label{lemma: filtration of assembly}
$\Rcal_{m}\left[F(\one)\wedge X \right]
     =\Rcal_{m}\left[F(\one) \right]\wedge X$.
\end{lemma}

\begin{proof}
We consider the following diagram of augmented $\Gamma$-spaces: 
\begin{equation}\label{eq: augmentation on assembly}
\begin{CD}
F(\one)\wedge X @>>> F(\one\wedge X)
         @.  \cong @.F(X)\\ 
@VVV  @VVV\\
\Sp^{\infty}(\one)\wedge X @>>> \Sp^{\infty}(\one\wedge X) 
         @.\ \cong\ @.\Sp^{\infty}(X).
\end{CD}
\end{equation}
The horizontal maps are assembly maps for the $\Gamma$-spaces
$F$ and $\Sp^{\infty}$, and the vertical maps are given by the
augmentation of $F$. The diagram strictly commutes, because the 
assembly map is natural in maps of $\Gamma$-spaces.

By definition, the augmentation of $F(\one)\wedge X$ is given by the
clockwise path from top left to bottom right, and thus
$\Rcal_{m}\left[F(\one)\wedge X \right]$ is defined as the inverse
image of $\Sp^{m}(X)\subseteq\Sp^{\infty}(X)$ by this path. However,
since \eqref{eq: augmentation on assembly} commutes, we can go the
other way around the square. As noted above, the inverse image of
$\Sp^{m}(X)$ in $\Sp^{\infty}(\one)\wedge X$ is 
$\Sp^{m}(\one)\wedge X $. Because the inverse image of
$\Sp^{m}(\one)\wedge X $ in $F(\one)\wedge X$ is, by definition,
$\Rcal_{m}\left[F(\one) \right]\wedge X$, the lemma follows. 
\end{proof}

Now we want to use the assembly map to compare two $\Gamma$-spaces. 
Let $F\rightarrow G$ be an augmentation-preserving map of 
augmented, very special $\Gamma$-spaces. Then we have
a commutative diagram 
\[
\begin{CD}
\Rcal_{m}\left[F(\one)\wedge X\right] @>>> \Rcal_{m}\left[F(X)\right] \\
@VVV @VVV \\
\Rcal_{m}\left[G(\one)\wedge X\right] @>>> \Rcal_{m}\left[G(X) \right].
\end{CD}
\]
By using Lemma~\ref{lemma: filtration of assembly} to replace
$\Rcal_{m}\left[F(\one)\wedge X\right]$ and 
$\Rcal_{m}\left[G(\one)\wedge X\right]$ with
$\Rcal_{m}[F(\one)]\wedge X$ and $\Rcal_{m}[G(\one)]\wedge X$,
respectively, we obtain the diagram of the following theorem,
which is our goal for this subsection. 

\begin{proposition}  \label{proposition: equivalent cats}
\propositionequivalentcatstext
\end{proposition}

Before tackling the proof of Proposition~\ref{proposition: equivalent cats}, we
establish a lemma about the righthand side of its diagram
in the case that $X$ is a finite pointed set. If $k\in\naturals$, then
because $F$ is very special, we have an equivalence 
$F(\kund) \rightarrow F(\one)^{k}$ induced by the $k$ collapse maps.  
We would like to understand $\Rcal_{m} [F(\kund)]$ in terms of $F(\one)^{k}$.

\begin{lemma}         \label{lemma: product filter}
Suppose that $F$ is an augmented very special $\Gamma$-space. Then
\begin{equation*}    
\Rcal_{m} [F(\kund)]
    \simeq \underset{i_{1}+\dots+i_{k}\leq m}{\colim}
       \Rcal_{i_{1}} [F(\one)]\times\dots\times \Rcal_{i_{k}} [F(\one)]
\end{equation*}
\end{lemma}

\begin{proof}
First consider the special case of symmetric powers. 
Let
\[
C:= \underset{i_{1}+\dots+i_{k}\leq m}{\colim}
       \Sp^{i_{1}}(\one)\times\dots\times \Sp^{i_{k}}(\one)
       \ \ \subset\Sp^{\infty}(\one)^{k},
\]
where the colimit is taken over the poset of $k$-tuples with
$(i_{1},\dots,i_{k})\leq (i'_{1},\dots,i'_{k})$ if $i_{j}\leq i'_{j}$
for all $j=1,\dots,k$. The isomorphism of discrete sets
$\Sp^{\infty}(\kund)\xrightarrow{\cong}\Sp^{\infty}(\one)^{k}$
gives us an isomorphism between $\Sp^{m}(\kund)$ and $C$. 

Then we consider the map of diagrams
\[
\left( \begin{CD}
@.  F(\kund) \\
@. @VVV\\
\Sp^m(\kund) @>>> \Sp^{\infty}(\kund)  \\
 \end{CD}\right)
\hspace{2.5em}\xrightarrow{\hspace{4em}}\hspace{2.5em}
\left( \begin{CD}
@.  F(\one)^{k}\\
@. @VVV\\
C @>>> \Sp^{\infty}(\one)^{k}
 \end{CD}\right),
 \]
which is an isomorphism on the lower row.
Indeed, each bottom row is an inclusion of discrete sets,
and the map of upper right corners is a homotopy equivalence, which is
sufficient to guarantee that the map between the (strict) pullbacks
is a homotopy equivalence. However, the pullback of the left
diagram is $\Rcal_{m}[F(\kund)]$ by definition, and the pullback
of the right diagram is 
$\underset{i_{1}+\dots+i_{k}\leq m}{\colim}
       \Rcal_{i_{1}} [F(\one)]\times\dots\times \Rcal_{i_{k}} [F(\one)]$. 
The lemma follows. 


\end{proof}

\begin{proof}[Proof of Proposition~\ref{proposition: equivalent cats}]
  It suffices to establish the proposition for $X=\kund$ for arbitrary
  $k\in\naturals$. Furthermore, since $\Rcal_{i}\Rcal_{m}=\Rcal_{i}$
  if $i\leq m$ and $\Rcal_{j}\Rcal_{m}=\Rcal_{m}$ if $j>m$, it
  actually suffices to show that for $i\leq m$ we have a homotopy
  pushout diagram
\begin{equation*}
\begin{CD}
\Rcal_{i}[F(\one)]\wedge\kund @>>> \Rcal_{i} [F(\kund)]\\
@VVV @VVV\\
\Rcal_{i}[G(\one)]\wedge\kund @>>> \Rcal_{i} [G(\kund)].
\end{CD}
\end{equation*}
We will deal with the case $i=m$ in detail. The case $i<m$ is proved
in the same way, except it is easier.

By Lemma~\ref{lemma: product filter}, we need to show that the
following is a homotopy pushout:
\begin{equation}
\begin{CD}\label{eq: key pushout translated}
\Rcal_{m}[F(\one)]\wedge \kund @>>> 
   \underset{i_{1}+\dots+i_{k}\leq m}{\colim}
       \Rcal_{i_{1}} [F(\one)]\times\dots\times \Rcal_{i_{k}} [F(\one)]\\
@VVV @VVV\\
\Rcal_{m}[G(\one)]\wedge \kund @>>> 
   \underset{i_{1}+\dots+i_{k}\leq m}{\colim}
       \Rcal_{i_{1}} [G(\one)]\times\dots\times \Rcal_{i_{k}} [G(\one)].
\end{CD}
\end{equation}
As in the proof of Lemma~\ref{lemma: filtered wedge and product}, we
note that the diagram over which the colimits are taken is cofibrant,
so the colimits can be replaced by homotopy colimits \cite{D-S}.
Further, the left column consists of $k$-fold wedge sums, and these
can be written as homotopy colimits over the same category by taking
$(i_{1},\dots,i_{k})$ to a point if $i_{j}<m$ for all $i_{j}$, and
taking $(0,\dots,0,m,0,\dots,0)$ to $\Rcal_{m}[F(\one)]$ (for the
upper row) or $\Rcal_{m}[G(\one)]$ (for the lower row).

With this setup, we conclude that the proposition follows because 
\eqref{eq: key pushout translated} is a homotopy colimit over the
category of $k$-tuples of the two following types of squares, all of which
are themselves homotopy pushout squares:
\begin{enumerate}
\item when all $i_{j}<m$,
\[
\begin{CD}
* @>>> \Rcal_{i_{1}} [F(\one)]\times\dots\times \Rcal_{i_{k}} [F(\one)]\\
@VV{=}V @VV{\simeq}V\\
* @>>> \Rcal_{i_{1}} [G(\one)]\times\dots\times \Rcal_{i_{k}} [G(\one)]
\end{CD}
\]
\item if some $i_{j}=m$,
\[
\begin{CD}
\Rcal_{m}[F(\one)]@>>{\cong}> 
      \Rcal_{0}[F(\one)]\times\dots\times
              \Rcal_{m}[F(\one)]\times\dots\times \Rcal_{0} [F(\one)]\\
@VVV @VVV\\
\Rcal_{m}[G(\one)]@>>{\cong}> 
      \Rcal_{0}[G(\one)]\times\dots\times
              \Rcal_{m}[G(\one)]\times\dots\times \Rcal_{0} [G(\one)].
\end{CD}
\]
\end{enumerate}
\end{proof}

In Section~\ref{section: comparisons}, we apply a special case
of Proposition~\ref{proposition: equivalent cats}, which we single out as
a corollary. First we need to introduce one more definition.

\begin{definition}
  Let $F$ be an augmented very special $\Gamma$-space.  We say that $F$
  is $m$-reduced if the following natural map is a weak homotopy
  equivalence:
\[
\Rcal_{m}\left[F(\one)\right]
   \xrightarrow{\ \simeq\ } \Rcal_{m}\left[\Sp^{\infty}(\one)\right].
\] 
\end{definition}

Since $\Sp^{\infty}(\one)=\naturals$, this is saying that the first
$m$ pieces (usually components) of $F(\one)$ are contractible.  For
example, if $F$ is the Segal construction for the category of finite
pointed sets, then $F(\one)\simeq\coprod_{i} B\Sigma_{i}$, and $F$ is
$1$-reduced because $B\Sigma_{1}\simeq *$.  An augmented
$\Gamma$-space is always $0$-reduced by definition.

Since $\Rcal_{m}[F(X)]$ is defined as the inverse image
of $\Sp^{m}(X)$ under the augmentation $F(X)\rightarrow\Sp^{\infty}(X)$,
and $\Sp^{\infty}(\one)$ is discrete, we see that 
\[
\Rcal_{m}\left[F(\one)\right]\cong\bigvee_{1\leq i\leq m}
      \Rcal_{i}/\Rcal_{i-1}\left[F(\one)\right].
\]
We can use this splitting to define the maps in the left square of the 
following diagram:
\begin{equation}   \label{eq: two pushouts}
\begin{CD}
\Rcal_{m}/\Rcal_{m-1}\left[F(\one)\right]\wedge X @>>>
\Rcal_{m}\left[F(\one)\right]\wedge X @>>> \Rcal_{m}[F(X)]\\
@VVV @VVV @VVV\\
\Rcal_{m}/\Rcal_{m-1}\left[\Sp^\infty(\one)\right]\wedge X 
       @>>>\Rcal_{m}\left[\Sp^{\infty}(\one)\right]\wedge X 
        @>>> \Sp^{m} (X).
\end{CD}
\end{equation}
However, $\Rcal_{m}/\Rcal_{m-1}\left[\Sp^\infty(\one)\right]$
is just $S^0$, and so we arrive at the following corollary
of Proposition~\ref{proposition: equivalent cats}. 

\begin{corollary}\label{cor: basic pushout}
  Suppose that $F$ is an augmented very special $\Gamma$-space that
  is $(m-1)$-reduced. Then there is a homotopy
  pushout diagram of augmented $\Gamma$-spaces
\[
\begin{CD}
\Rcal_{m}/\Rcal_{m-1}\left[F(\one)\right] \wedge X 
      @>>> \Rcal_{m}F (X)\\
@VVV @VVV\\
X @>>> \Sp^{m} (X),
\end{CD}
\]
where the identity $\Gamma$-space in the lower left corner is given 
the augmentation $X\rightarrow\Sp^\infty X$ by $x\mapsto mx$.
\end{corollary}

\begin{proof}
  The left square of \eqref{eq: two pushouts} is a homotopy pushout
  because $F$ is $(m-1)$-reduced, and the right square is a homotopy
  pushout by direct application of 
  Proposition~\ref{proposition: equivalent cats}.  Thus the outer
  square is a homotopy pushout, and the corollary follows.
\end{proof}

\begin{remark}   \label{remark: supersedes}
Proposition~\ref{proposition: equivalent cats} is closely related to 
Theorem~1.1 of \cite{Lesh-TAMS}, the proof of which 
comes in two halves, \cite{Lesh-TAMS} Lemma~6.1 and Lemma~6.2. 
Proposition~\ref{proposition: equivalent cats} is a slightly strengthened
version of \cite{Lesh-TAMS} Lemma~6.1. We also obtain a version of the 
second half: Lemma~\ref{lemma: first stage pushout} of this paper contains
Lemma~6.3 of \cite{Lesh-TAMS} in the case of the family of all subgroups,
but it is not as general as the original lemma. 
  

\end{remark}

\section{Comparisons}
\label{section: comparisons}

Our goal in this section is to establish the precise relationship of
the modified stable rank filtration of a $K$-theory spectrum $\kay\Ccal$,
defined in Section~\ref{section: filtered Gamma spaces}, to two other 
filtrations of $\kay\Ccal$: the filtration defined by the authors
in \cite{A-L}, and the stable rank filtration of a $K$-theory spectrum
that was defined by Rognes in \cite{Rognes}. 
In Section~\ref{section: comparison of constructions}, we do the first
of these comparisons, and we show that the filtration constructed by the
authors in \cite{A-L} is related to the modified stable rank filtration
through a pushout diagram involving the symmetric powers of the sphere
spectrum and the spectrum $\kay\Ccal$ itself. This allows us to verify 
the claim made in the introduction that the spectra $\Mbar(k)$ in 
the reduced $bu$-Whitehead Conjecture are in fact the subquotients
of the modified stable rank filtration. 

Then in Section~\ref{section: comparison modified to Rognes}, we
compare the modified stable rank filtration to the original
construction of the stable rank filtration by Rognes. There is a map
from the first to the second, which need not be, in general, an equivalence
of filtrations in the case of the algebraic $K$-theory of a discrete
ring.  However, we show that it is an equivalence in the case of
topological $K$-theory. Thus, in this case the modified stable rank
filtration provides another model for Rognes's filtration.  In
combination with the results of 
Section~\ref{section: comparison of constructions}, we find that the
spectra $\Mbar(k)$ introduced in Section~\ref{section: Mbar} are actually
filtration quotients in the stable rank filtration of Rognes.

\subsection{Comparing the modified stable rank filtration 
to~\cite{A-L}}
\label{section: comparison of constructions}
Let $\Ccal$ be an augmented permutative category. Our previous work in
\cite{A-L}, and the modified rank filtration constructed in
Section~\ref{section: filtered Gamma spaces}, give two filtered
$\Gamma$-spaces that we can associate with $\Ccal$, and our goal in
this subsection is to compare them. The sequence
$(\Kcal_{m}\Ccal)(-)$ constructed in~\cite{A-L} is a sequence of
augmented, very special $\Gamma$-spaces interpolating beween the
$\Gamma$-spaces $\Ccal(-)$ and $\naturals(-)=\Sp^{\infty}(-)$. 
The stabilization of this sequence is a sequence of spectra beginning with
$\kay\Ccal$ and ending with $\HZ$. In the case that $\Ccal$ is the
category of finite-dimensional complex vector spaces, these are the
spectra whose subquotients appear in the $bu$ Whitehead
conjecture (Conjecture~\ref{conjecture: bu Whitehead}).  On the other
hand, we have the modified rank filtration, i.e., the augmented
$\Gamma$-spaces $\Rcal_{m}[\Ccal(-)]$ constructed in
Section~\ref{section: filtered Gamma spaces}. This is a sequence of
augmented, but not very special $\Gamma$-spaces interpolating between
$*$ and $\Ccal(-)$.

Certainly the sequences $(\Kcal_{m}\Ccal)(-)$ and
$\Rcal_{m}[\Ccal(-)]$ cannot be the same. The spectra $A_{m}$ that are
the stabilizations of $(\Kcal_{m}\Ccal)(-)$ go from $\kay\Ccal$ to
$\HZ$, while the stabilizations $\Rcal_{m}\kay\Ccal$ of
$\Rcal_{m}[\Ccal(-)]$ go from $*$ to $\kay\Ccal$. However, the main
technical result of this subsection 
(Theorem~\ref{theorem: connecting A-L with Rognes}) has an immediate
consequence (Theorem~\ref{thm: pushout theorem}) giving a homotopy
pushout diagram of spectra
\begin{equation} \label{eq: hoped for}
\begin{CD}
\Rcal_{m}\kay\Ccal @>>> \kay\Ccal\\
@VVV @VVV\\
\Sp^{m}(\Sphere) @>>> A_{m},
\end{CD}
\end{equation}
a result that is
heuristically plausible because the vertical fibers on the left go
from $*$ to the fiber of the augmentation $\kay\Ccal \rightarrow \HZ$,
and so do the vertical fibers on the right. This gives a precise
stable relationship between the modified stable rank filtration in
the upper left corner of~\eqref{eq: hoped for}, 
and the filtration defined in \cite{A-L} in
the lower right corner, though the result is only true stably and not
on the level of the corresponding $\Gamma$-spaces.

To set up the argument to obtain~\eqref{eq: hoped for}, for each $m$
we define an augmented $\Gamma$-space $\Ecal_{m}[\Ccal(X)]$ by an
objectwise pushout diagram
\begin{equation}  \label{diag: define E_m}
\begin{CD}
\Rcal_m[\Ccal(X)] @>>> \Ccal(X)\\
@VVV @VVV\\
\Sp^{m}(X) @>>> \Ecal_{m}[\Ccal(X)],
\end{CD}
\end{equation}
where the maps $\Rcal_m[\Ccal(X)]\rightarrow\Ccal(X)$ and
$\Rcal_m[\Ccal(X)]\rightarrow\Sp^{m}(X)$ are the maps that define
$\Rcal_m[\Ccal(X)]$ as a pullback. It is easy to check that 
\eqref{diag: define E_m} is
a strong pushout diagram of augmented $\Gamma$-spaces, that is, that
the diagram remains a pushout when any $\Rcal_{i}$ is applied to it.


\begin{theorem}   \label{theorem: connecting A-L with Rognes}
For each $m$, there is a chain of augmented stable equivalences
\[
\Ecal_{m}[\Ccal(-)] \simeq \Kcal_{m}[\Ccal(-)].
\]
The stable equivalences commute with the maps from the $(m-1)$-st stage to the
$m$-th stage on both sides.
\end{theorem}

To obtain diagram~\eqref{eq: hoped for} from the theorem, note that
diagram~\eqref{diag: define E_m} is a homotopy pushout as well as a
pushout diagram, because the top map is an objectwise cofibration.
The following theorem is then an immediate corollary, and is the main
result of this section. 

\begin{theorem}    \label{thm: pushout theorem}
\ConnectionTheoremText
\end{theorem}

Theorem~\ref{thm: pushout theorem} implies that
\[
 \Rcal_{m}\kay\Ccal/\Rcal_{m-1}\kay\Ccal 
      \simeq \Sigma^{-1}
          \left(A_{m}  /\Sp^{m}(\Sphere)\right)
        / \left(A_{m-1}/\Sp^{m-1}(\Sphere)\right).
\]
When $m=p^{k}$, the spectrum on the right is the spectrum $\Mbar(k)$,
by definition from~\eqref{eq: defn Mbar(k)}. Thus
the spectra $\Mbar(k)$
are, in fact, subquotients of the modified rank filtration.

\begin{corollary} \label{cor: A-L vis Rognes, revisited}
\ConnectionToMbarText
\end{corollary}

The proof of Theorem~\ref{theorem: connecting A-L with Rognes} is by
induction on $m$, the case $m=0$ being obvious because in that case
both sides are $\Ccal(-)$. The main idea of the proof is that the
augmented $\Gamma$-spaces $\Ecal_m\Ccal(-)$ and $\Kcal_m\Ccal(-)$ are
obtained from $\Ecal_{m-1}\Ccal(-)$ and $\Kcal_{m-1}\Ccal(-)$, respectively,
by means of the same inductive formula.
  
We first analyze the construction $\Kcal_m\Ccal$. Recall that $\Ccal$
has an augmentation $\epsilon\colon\Ccal \longrightarrow \naturals$.
We write $\Ccal_{m}$ for $\epsilon^{-1}(m)$, and we call $\Ccal$
$m$-reduced if $|\Ccal_{i}|\simeq *$ for $i\leq m$. (An augmented
permutative category $\Ccal$ is always $0$-reduced by
Definition~\ref{defn: augmented category}.)
The inductive construction of \cite{A-L}
begins with $\Kcal_{0}\Ccal:=\Ccal$, and the inductive step is
equivalent to taking the augmented permutative category
$\Kcal_{m-1}\Ccal$, which is $(m-1)$-reduced, and using a bar
construction to ``kill'' its $m$-th component,
$\left(\Kcal_{m-1}\Ccal\right)_{m}$, thereby obtaining the $m$-reduced
augmented permutative category $\Kcal_{m}\Ccal$ (Construction~3.8 and
Proposition 2.5 of \cite{A-L}). 
Theorem~3.9 of \cite{A-L} establishes a stable homotopy
pushout diagram involving the associated spectra, and the following
lemma strengthens this result by using Section~\ref{sec: sums and products}
of the current work to show that the homotopy
pushout diagram in question is still a stable homotopy pushout at
each level of the modified stable rank filtration,
i.e., we have a {\emph{strong}} augmented stable homotopy pushout diagram.

\begin{lemma}\label{lemma: out with the old}
There is a strong stable homotopy pushout square of augmented $\Gamma$-spaces
\[
\begin{CD}
B\left(\Kcal_{m-1}\Ccal\right)_{m}\mbox{}_{+} \wedge X 
           @>>> \left(\Kcal_{m-1}\Ccal\right)(X)\\
@VVV @VVV \\
X @>>>  \left(\Kcal_{m}\Ccal\right)(X),
\end{CD}
\]
where the augmentation $X\rightarrow\Sp^{\infty}(X)$ of the lower left corner
is given by $x\mapsto mx$.
\end{lemma}

\begin{proof}
  Recall from Definition~3.1 of \cite{A-L} 
  that if $\Dcal$ is a small category, then
  $\Free(\Dcal)$ is the free permutative category generated by
  $\Dcal$. By \cite{A-L} Proposition~2.5 and Construction~3.8 ,
  $\Kcal_{m}\Ccal$ is equivalent to the (augmented) bar construction
\[
\Kcal_{m}\Ccal \simeq \BAR\biggl(\Free\{m\}, 
                 \Free\left(\Kcal_{m-1}\Ccal\right)_{m},
                          \Kcal_{m-1}\Ccal
                     \biggr).
\]
It now follows by Proposition~\ref{proposition: two pushouts} that
there is a strong augmented stable homotopy pushout square
\[
\begin{CD}
\Free\left(\Kcal_{m-1}\Ccal\right)_{m}( X )
           @>>> \Kcal_{m-1}\Ccal(X)\\
@VVV @VVV \\
\left(\Free\{m\}\right)(X) @>>>  \Kcal_{m}\Ccal(X).
\end{CD}
\]
Now consider the diagram
\[
\begin{CD}
\left(\Kcal_{m-1}\Ccal\right)_{m}\mbox{}_{+} \wedge X 
           @>>> \Free\left(\Kcal_{m-1}\Ccal\right)_{m}( X )
           @>>> \Kcal_{m-1}\Ccal(X)\\
@VVV @VVV @VVV \\
X @>>> \left(\Free\{m\}\right)(X) @>>>  \Kcal_{m}\Ccal(X)
\end{CD}
\]
where the horizontal maps are given by the natural inclusions. We just
saw that the right square is a strong augmented stable homotopy pushout. The 
horizontal maps on the left are stable equivalences
because they have the form $Y\mapsto QY$, from~\cite{A-L}, 
Proposition~3.3; this map is
a strong augmented stable equivalence with the standard augmentation, and
the augmentation here in the bottom row 
is simply the standard one multiplied by $m$. 

It follows that the outer square is a strong augmented
stable homotopy pushout, which is what we wanted to prove.
\end{proof}

\begin{lemma}   \label{lemma: first stage pushout}
There is a strong stable homotopy pushout square of augmented
$\Gamma$-spaces
\[
\begin{CD}
\Rcal_{m}\left[\Kcal_{m-1}\Ccal(X)\right] @>>> \Kcal_{m-1}\Ccal(X)\\
@VVV @VVV \\
\Rcal_{m}\left[\Kcal_{m}\Ccal(X)\right] @>>> \Kcal_{m}\Ccal(X).
\end{CD}
\]
\end{lemma}

\begin{proof}
  We tack the desired square onto the square obtained by
  applying Corollary~\ref{cor: basic pushout} to the $(m-1)$-reduced
  augmented category $\Kcal_{m-1}\Ccal$:
\begin{equation}
\begin{CD} \label{eq: two squares} 
\left(\left(\Kcal_{m-1}\Ccal\right)_{m}\right)_+\wedge X 
     @>>>\Rcal_{m}\left[\Kcal_{m-1}\Ccal(X)\right] @>>> \Kcal_{m-1}\Ccal(X)\\
@VVV @VVV @VVV \\
X  @>>>\Rcal_{m}\left[\Kcal_{m}\Ccal(X)\right] @>>> \Kcal_{m}\Ccal(X)
\end{CD}
\end{equation}
The left square is a strong augmented stable homotopy pushout square by
Corollary~\ref{cor: basic pushout}, since $\Kcal_{m}\Ccal(X)$ being $m$-reduced
implies that 
$\Rcal_{i}\left[\Kcal_{m}\Ccal(X)\right]
       \simeq\Rcal_{i}\left[\naturals(X)\right]=\Sp^{i}(X)$
for $i\leq m$. 
The outer square of~\eqref{eq: two squares} is a strong augmented
stable homotopy pushout square by Lemma~\ref{lemma: out with the old}.
Hence the right square of~\eqref{eq: two squares}
is a strong augmented stable homotopy pushout square.
\end{proof}

\begin{proof}[Proof of Theorem~\ref{theorem: connecting A-L with Rognes}]
  
Our plan is to establish an inductive formula for $\Ecal_m\Ccal$
that corresponds to the formula provided by 
Lemma~\ref{lemma: first stage pushout} for $\Kcal_m\Ccal$.

Consider the following diagram of augmented $\Gamma$-spaces
\begin{equation}\label{eq: double decker}
\begin{CD}
\Rcal_m\left[\Ccal(X)\right] @>>> \Ccal(X)\\
@VVV @VVV \\
\Rcal_m\left[\Ecal_{m-1}\Ccal(X)\right] @>>> \Ecal_{m-1}\Ccal(X) \\
@VVV @VVV\\
\Sp^{m}(X) @>>> \Ecal_{m}\Ccal(X).
\end{CD}
\end{equation}
By definition from diagram~\eqref{diag: define E_m}, the outer square
is an augmented pushout square, and it remains so after application of
$\Rcal_{i}$, so it is a strong augmented pushout square.
The top horizontal map is a cofibration,
so the outer square is also a strong augmented homotopy pushout square.
Therefore $\Rcal_m\left[\Ecal_{m}\Ccal(X)\right]\simeq\Sp^{m}(X)$. 

If we prove that the bottom square is an augmented stable homotopy
pushout square, then we will have an inductive formula for
$\Ecal_{m}\Ccal(X)$ that matches the one for $\Kcal_{m}\Ccal(X)$. To
accomplish this, it is enough to prove that the upper square is a
strong augmented stable homotopy pushout square.

Diagram~\eqref{diag: define E_m} for $m-1$ gives us a strong
augmented pushout square
\[
\begin{CD}
\Rcal_{m-1}\left[\Ccal(X)\right] @>>> \Ccal(X)\\
@VVV @VVV \\
\Sp^{m-1}(X) @>>> \Ecal_{m-1}\Ccal(X),
\end{CD}
\]
and applying $\Rcal_{m}$ to it and using the fact that
$\Rcal_{m}\Rcal_{m-1}=\Rcal_{m-1}$ gives the strong augmented pushout square
\[
\begin{CD}
\Rcal_{m-1}\left[\Ccal(X)\right] @>>> \Rcal_{m}\left[\Ccal(X)\right]\\
@VVV @VVV \\
\Sp^{m-1}(X) @>>> \Rcal_{m}\left[\Ecal_{m-1}\Ccal(X)\right]. 
\end{CD}
\]
Consider the diagram
\[
\begin{CD}
\Rcal_{m-1}\left[\Ccal(X)\right] @>>> \Rcal_m\left[\Ccal(X)\right] 
                                 @>>> \Ccal(X)\\
@VVV @VVV @VVV \\
\Sp^{m-1}(X)
                               @>>> \Rcal_m\left[\Ecal_{m-1}\Ccal(X)\right] 
                               @>>> \Ecal_{m-1}\Ccal(X) \\
\end{CD}
\]
The outer square and left square are augmented pushouts, respectively
by definition and by the above discussion.  It follows that the right
square is an augmented pushout, and since the top map is a
cofibration, it is also a homotopy pushout.  Thus the upper square
of~\eqref{eq: double decker} is a homotopy pushout, and therefore
the lower square is likewise.

To prove the theorem, suppose given (by induction) a
strong augmented stable equivalence 
$\Kcal_{m-1}\Ccal\rightarrow\Ecal_{m-1}\Ccal$. 
The induced map
$\Rcal_{m}\left[K_{m-1}\Ccal\right]
        \rightarrow\Rcal_{m}\left[\Ecal_{m-1}\Ccal\right]$
is also a strong augmented stable equivalence, and we have the following
diagram in which all the vertical maps are strong augmented stable equivalences,
in the case of the leftmost arrow by Corollary~\ref{cor: basic pushout}
because $\Kcal_{m}\Ccal$ is $m$-reduced:
\[
\begin{CD}
\Rcal_{m}\left[\Kcal_{m}\Ccal\right] @<<< \Rcal_{m}\left[K_{m-1}\Ccal\right] 
            @>>> \Kcal_{m-1}\Ccal\\
@VVV @VVV @VVV\\
\Sp^{m}                 @<<< \Rcal_{m}\left[\Ecal_{m-1}\Ccal\right] 
            @>>>  \Ecal_{m-1}\Ccal.
\end{CD}
\]
As a consequence, there is an augmented stable equivalence from
the homotopy pushout of the top row to the homotopy
pushout of the bottom row. But that is a stable augmented
equivalence from $\Kcal_{m}\Ccal$ to $\Ecal_{m}\Ccal$, as required. 

\end{proof}

\subsection{\protect Comparing stable rank filtrations}
\label{section: comparison modified to Rognes}

In the previous section, we related two filtrations of a $K$-theory
spectrum $\kay\Ccal$: the modified stable rank filtration constructed
in Section~\ref{section: filtered Gamma spaces}, and the filtration of
$\kay\Ccal$ constructed in~\cite{A-L}. In this section, we relate the
construction of Section~\ref{section: filtered Gamma spaces} to the
original stable rank filtration of Rognes constructed in
\cite{Rognes}.  We also justify our use of the terminology ``modified
stable rank filtration'' by explaining how the two filtrations come
from the same idea applied to two different infinite loop space
machines.

Our plan is to use a comparison of Segal's and Waldhausen's $K$-theory
constructions to describe a canonical map from the modified stable
rank filtration to the original stable rank filtration. In general,
this map need not be an equivalence of filtrations. In particular, if the
category being considered is the category of finite-dimensional free
modules over a ring $R$ satisfying dimension invariance, then the
comparison of filtration quotients amounts to including the set of
diagonal matrices into the set of upper triangular matrices. However,
for matrices over a contractible topological ring, such as $\reals$ or
$\complexes$, this inclusion is in fact a homotopy equivalence.  Thus
in the special case of complex topological $K$-theory, we establish
the following equivalence. (The proposition also applies to real 
topological $K$-theory.)

\begin{proposition}\label{prop: equivalence of Rognes filtrations}
\PropositionEquivalenceText
\end{proposition}

As a consequence, the homotopy pushout square of spectra in
Theorem~\ref{thm: pushout theorem}, when applied to topological
complex $K$-theory, gives us a good understanding of Rognes's
stable rank filtration for this case.  In particular, Rognes
conjectures in \cite{Rognes} that the subquotient 
$F_{m}\kay R/F_{m-1}\kay R$ \ is \ $(2m-3)$-connected for a large
class of rings $R$, and we establish that this connectivity conjecture
actually holds sharply for infinitely many filtration quotients of
topological complex $K$-theory.

\begin{proposition}  \label{prop: Rognes connectivity conjecture}
\ConnectivityPropositionText
\end{proposition}

To begin the work of comparing the modified stable rank filtration of
Definition~\ref{defn: modified stable rank filtration} and the
original stable rank filtration defined by Rognes, we recapitulate
from \cite{Waldhausen}, Section~1.8 some elements of the comparison
between Segal's $\Gamma$-space construction and Waldhausen's
$S_{\bullet}$ construction.  This is because the modified and original
stable rank filtrations are based, respectively, on filtrations of
these constructions.

We reformulate slightly, in two stages. First, an alternate
(isomorphic) construction of the Segal $K$-theory of $\Ccal$ comes
from thinking of $S^{k}$ as the $k$-simplicial set
$S^{1}\wedge\dots\wedge S^{1}$, evaluating the nerve of $\Ccal(-)$
levelwise to obtain a $k$-simplicial space, and then taking the
geometric realization to obtain the $k$-th space in the $K$-theory
spectrum of $\Ccal$. (Note that for $k=1$, a $k$-simplicial set is
just a classical simplicial set, and for $k=2$, it is a bisimplicial
set, or simplicial space.)  Second, for a pointed set $S$, the
category $\Ccal(S)$ is again a category with sums. That is,
it is a category to which Segal's construction may be
applied.  Thus we can iterate by applying Segal's construction with
the simplicial category $\Ccal(S^{1})$ to the simplicial set $S^{1}$
to obtain a bisimplicial category. The following lemma formalizes the
equivalence of these constructions.

\begin{lemma}
The natural map $\Ccal(S\wedge T)\rightarrow \Ccal(S)(T)$
is an equivalence of categories. 
\end{lemma}
\noindent 
(The lemma follows from the fact that both categories are equivalent to 
the product category $\prod_{S\wedge T}\Ccal$.)  

\begin{corollary}
The natural map 
$\Ccal\left(S^{1}\wedge S^{1}\right)
     \rightarrow [\Ccal(S^{1})](S^{1})$ 
of bisimplicial sets is an equivalence on geometric realizations. 
\end{corollary}

By induction, we conclude that we can construct the $k$-th space in 
the Segal $K$-theory spectrum of $\Ccal$ by iterating Segal's construction
for the category $\Ccal$: we take the $k$-simplicial category with sums
that we have at the $k$-th stage and use it to evaluate Segal's construction 
on the simplicial set $S^{1}$ to obtain a $(k+1)$-simplicial 
category. This is a useful phrasing because Waldhausen
describes his $K$-theory of $\Ccal$ in terms of an iterated construction. 

We summarize Waldhausen's construction next. Recall that if $\Ccal$
is a category with cofibrations and weak equivalences, then 
$wS_\bullet\Ccal$ is a simplicial category whose objects 
are determined by data of the following form, 
where $\rightarrowtail$ denotes a cofibration in $\Ccal$:
\[
(B_1\rightarrowtail\cdots\rightarrowtail B_q, \mbox{ choices of subquotients}).
\]
Morphisms are pointwise weak equivalences between such objects.
See~\cite{Waldhausen} for more details.  If $\Ccal$ is a category with
cofibrations and weak equivalences, then $S_\bullet\Ccal$ is again a
category with cofibrations and weak equivalences, and thus one may
iterate the $S_\bullet$ construction $k$ times to obtain the $k$-simplicial
category $S_\bullet^k\Ccal$. Considering only the morphisms that are weak
equivalences gives us the $k$-simplicial category $wS_\bullet^k\Ccal$:
the objects of the
category $wS_{q_1}\ldots S_{q_k} \Ccal$ are cofibrant $q_1\times
q_2\times \ldots \times q_k$-arrays of objects in $\Ccal$, together
with choices of certain subquotients, and morphisms are objectwise
weak equivalences of such arrays. As with the Segal construction, we
take the sequence $w\Ccal, wS_\bullet\Ccal, \ldots,
wS_{\bullet}^{k}\Ccal, \ldots$ and obtain the prespectrum that is
Waldhausen's $K$-theory of $\Ccal$ by applying nerves levelwise and
then taking the geometric realization of the resulting $k$-simplicial
spaces. This prespectrum is in fact an $\Omega$-spectrum after the
first map.

To compare the modified stable rank filtration to the original stable
rank filtration, we need a map between Segal's construction and
Waldhausen's construction. 
Recall from Section~\ref{section: filtered Gamma spaces} that for
Segal's $\Gamma$-category $S\mapsto\Ccal(S)$, the objects of $\Ccal(S)$
are pairs
$(f,\alpha_{*})$, where $f$ is a function taking pointed subsets of $S$ to 
objects of $\Ccal$, and $\alpha_{*}$ is a collection of compatible isomorphisms 
$\alpha_{S_{1},S_{2}}:
       f(S_{1})\oplus f(S_{2})\xrightarrow{\cong} f(S_{1}\vee S_{2})$. 
If we take the simplicial model for $S^{1}$ that has the set $\q$ in
simplicial dimension $q$, then there is a functor sending the
$q$-simplices $(f,\alpha_{*})$ of $\Ccal(S^{1})$ to $wS_{q}\Ccal$ as
follows. Let $A_{i}$ denote $f(\{i\})$ and let $B_{i}=f(\{1,\dots, i\})$, 
where $\{i\}$ and $\{1,\dots, i\}$ implicitly contain the basepoint.
Then $(f,\alpha_{*})$ provides structure maps
\[
\alpha_{\{1,\dots,i\},\{i+1\}}: B_{i}\oplus A_{i+1} \longrightarrow B_{i+1}
\]
and the object $(f, \alpha_{*})$ in the Segal category is sent to the
object in the $S_{\bullet}$ construction
\[
B_{1} \rightarrowtail B_{2} \rightarrowtail 
       \ldots\rightarrowtail B_{q},
\]
together with the choices of quotient maps given by the inverse of the 
isomorphisms from $\alpha_{*}$ followed by projection: 
$B_{i+1}\xrightarrow{\cong} B_{i}\oplus A_{i+1} \epi A_{i+1}$. 

This simplicial functor generalizes to a $k$-simplicial functor that
for each $k$ takes the $k$-fold iterate of Segal's construction for
$\Ccal$ on $S^{1}$ to the $k$-fold iterate $wS_\bullet^k \Ccal$. Hence
it induces a $k$-simplicial functor
\[
 \Ccal(S^{k}) \longrightarrow wS_\bullet^k \Ccal.
\]
Together these simplicial functors induce a map from Segal's
$K$-theory to Waldhausen's $K$-theory, which is a homotopy equivalence
when cofibrations in $\Ccal^{n}$ are ``splittable up to weak 
equivalence'' (\cite{Waldhausen} Theorem~1.8.1). 
 
The example that is most relevant to us is the category of
free modules over a ring $R$ satisfying the dimension invariance property.
In this case, both constructions give a model for the free $K$-theory
of $R$ that coincides with Quillen's $K$-theory above dimension
zero  (\cite{Waldhausen}, Sections~8 and~9). 
Here we include the case when $R$ is $\reals$ or $\complexes$,
in which case $\Ccal$ is the topologically enriched category of (real
or complex) vector spaces, and both constructions give a model for
connective topological $K$-theory.
 
We need to understand the two stable rank filtrations in these terms.
Recall from Section~\ref{section: filtered Gamma spaces}
that the modified rank filtration of an augmented
$\Gamma$-space $F$ is defined as the pullback of the filtration of
$\Sp^\infty$ by $\{\Sp^m\}_{m=0}^\infty$. On $\Ccal(S^{k})$, this
amounts to filtering by the maximal dimension of the modules appearing
in an object $(f, \alpha_{*})$. Similarly, the original stable rank
filtration is obtained by filtering $wS_\bullet^k\Ccal$ by the maximal
dimension of the modules in the $k$-dimensional array of cofibrations
(see~\cite{Rognes} for more details). It is easy to see that the map
from Segal's spectrum to Waldhausen's spectrum respects the two
filtrations.

Now we are ready to prove 
Propositions~\ref{prop: equivalence of Rognes filtrations} 
and~\ref{prop: Rognes connectivity conjecture}.

\begin{proof}
  [Proof of Proposition~\ref{prop: equivalence of Rognes filtrations}]
  Let $R$ be $\reals$ or $\complexes$. The map of $k$-simplicial
  categories 
\begin{equation}     \label{eq: map of categories}
\Ccal(S^{k})\rightarrow wS_\bullet^k\Ccal
\end{equation}
respects the filtrations (modified rank filtration in the domain, rank
filtration in the codomain).  Therefore it is sufficient to check that it
induces a homotopy equivalence on the filtration quotients.

Both the domain and the codomain are $k$-simplicial groupoids, so they
are equivalent, in each dimension, to a disjoint union of groups, and
the map \eqref{eq: map of categories} is a map of $k$-simplicial
groupoids. The category $\Ccal(S^{k})$ is equivalent, in total
dimension exactly $n$, to the disjoint union of automorphism groups of
certain ordered direct-sum decompositions of $R^n$. Such a group is
equivalent to a group of invertible block diagonal matrices. The
category $wS_\bullet^k\Ccal$ is equivalent, in each dimension, to a
disjoint union of groups of automorphisms of a certain lattice of
subspaces of $R^n$ for various $n$. Such a group is equivalent to a
certain group of invertible block upper triangular matrices.  In these
terms, the map $\Ccal(S^{k})\rightarrow wS_\bullet^k\Ccal$ amounts, on
each connected component of the filtration quotients, to the inclusion
of a group of block diagonal matrices into a corresponding group of
block upper triangular matrices. In the topological case, such an
inclusion is always a homotopy equivalence. Thus the map
$\Ccal(S^{k})\rightarrow wS_\bullet^k\Ccal$ induces a homotopy
equivalence on each connected component of each filtration quotient.
Therefore, the whole map between $K$-theory spectra is a homotopy
equivalence.
 \end{proof}

As a consequence, we may verify Rognes's connectivity conjecture for
complex $K$-theory. Let $F_{m}bu$ be the $m$-th stage in the original stable
rank filtration of $bu$, as defined by Rognes.

 \begin{proof}[Proof of Proposition~\ref{prop: Rognes connectivity conjecture}]
  
By Proposition~\ref{prop: equivalence of Rognes filtrations}, we may
substitute the modified stable rank filtration for the original one,
that is, 
\[
F_{m}bu/F_{m-1}bu\simeq \Rcal_{m}/\Rcal_{m-1}bu. 
\]
We use the homotopy pushout square of spectra in 
Theorem~\ref{thm: pushout theorem} to find the connectivity of
$\Rcal_{m}bu/\Rcal_{m-1}bu$ by considering
$\Sp^{m}(\Sphere)/\Sp^{m-1}(\Sphere)$ and $\Sigma^{-1}A_{m}/A_{m-1}$.  It
follows from Theorems~9.5 and~9.7 of \cite{A-L} that
$\Sigma^{-1}A_m/A_{m-1}$ is $(2m-2)$-connected when $m=p^{k}$
and is contractible if $m\neq p^{k}$. On the other hand, we know that
$\Sp^m(\Sphere)/\Sp^{m-1}(\Sphere)$ is also contractible unless $m=p^k$ for
some prime $p$, and that a basis for the mod $p$ cohomology of 
$\Sp^{p^k}(\Sphere)/\Sp^{p^k-1}(\Sphere)$ is given by admissible words
of length $k$ in $\Acal/\Acal\beta$, where $\Acal$ is the Steenrod
algebra~\cite{Nakaoka}.
The lowest dimension of an admissible sequence of length $k$ 
is $2p^{k}-2$, which completes the proof of the proposition. 
\end{proof}


\subsection*{Acknowledgements}
The first author was partly supported by NSF grant DMS 0605073. 


\end{document}